\documentclass[11pt,english]{smfart}

\usepackage{amsmath,sfheaders,leftidx,amsthm}
\usepackage[all]{xy}
\usepackage{xspace,smfthm}
\usepackage[psamsfonts]{amssymb}
\usepackage[latin1]{inputenc}
\usepackage{graphicx,color,mathrsfs}
\usepackage{hyperref,fancyhdr,appendix}
\usepackage{xcolor}

\newtheorem*{remark}{\bf Remark}

\newtheorem{theorem}{\bf Theorem}[section]

\newtheorem{definition}[theorem]{\bf Definition}
\newtheorem{Theorem}{\bf Theorem}

\newtheorem{lemma}[theorem]{\bf Lemma}

\def\C{{\mathbb C}}

\def\R{{\mathbb R}}

\def\ddc{\mathrm{dd}^c}

\title{H\"older estimates and uniformity in arithmetic dynamics}
\author{Thomas Gauthier}
\address{Laboratoire de Math\'ematiques d'Orsay, B\^atiment 307, Universit\'e Paris-Saclay, 91405 Orsay Cedex, France}
\email{thomas.gauthier@universite-paris-saclay.fr}
\thanks{The author is partially supported by the Institut Universitaire de France.}

\begin{document}

\maketitle

\begin{abstract}
In this note we study common preperiodic points of rational maps of the Riemann Sphere. We show that given any degrees $d_1,d_2\geq2$, outside a Zariski closed subset of the space of pairs of rational maps $(f,g)$ of degree $d_1$ and $d_2$ respectively, the maps $f$ and $g$ share at most a uniformly bounded number of common preperiodic points. This generalizes a result of DeMarco and Mavraki to maps of possibly different degrees. Our main contribution is the use of H\"older properties of the Green function of a rational map to obtain height estimates.
\end{abstract}

\section{Introduction}
In a recent work, DeMarco and Mavraki~\cite{DeMarco_Mavraki_2024} showed that for any $d\geq2$, there is a constant $C\geq1$, depending only on $d$, such that a general pair of degree $d$ rational maps share at most $C$ preperiodic points. This question was first addressed by DeMarco, Krieger and Ye~\cite{DKY1,DKY2} in the case of the Latt\`es family and the quadratic polynomial family, with a slightly different proof, even though they relied on the same tool: the Arakelov-Zhang pairing (see also ~\cite{mavraki-schmidt} and \cite{Poineau-BFT} for works on this question).
We want here to give a proof of a generalization to the case when the rational maps can have different degrees - even multiplicatively independent degrees - which follows more or less the approach oy DeMarco, Krieger and Ye using a	 new ingredient: the H\"older regularity of the canonical metric invariant by an endomorphism of the projective space.


For any integers $d\geq2$, we denote by $\mathrm{Rat}_d$ the space of all endomorphisms of degree $d$ of $\mathbb{P}^1$, i.e.  the space of all degree $d$ rational maps on $\mathbb{P}^1$. This is a smooth affine variety defined over $\mathbb{Q}$ which can be identified with $\mathbb{P}^{2d+1}\setminus\{\mathrm{Res}=0\}$, where $\mathrm{Res}$ is the homogeneous MacCaulay resultant (see,~e.g.,~\cite{BB1}). For a given complex rational map $f:\mathbb{P}^1\to\mathbb{P}^1$, we denote
\[\mathrm{Preper}(f):=\{z\in \mathbb{P}^1(\mathbb{C})\, : \ \text{there are} \ n>m, \ \text{s.t.} \ f^n(z)=f^m(z)\}.\]
The main result of the present paper is the following.
\begin{Theorem}\label{tm:uniform}
For any $d_1,d_2\geq2$, there is a dense Zariski open subset $V\subseteq\mathrm{Rat}_{d_1}\times \mathrm{Rat}_{d_2}$ and an integer $N\geq1$ such that for any $(f,g)\in V(\mathbb{C})$,
\[\#\, \mathrm{Preper}(f)\cap\mathrm{Preper}(g)\leq N.\]
\end{Theorem}

As written above, to show this result, we follow the approach of \cite{DKY1,DKY2}: the key ingredient is the \emph{Arakelov-Zhang pairing} of two rational maps $f,g:\mathbb{P}^1\to\mathbb{P}^1$ defined over a number field $\mathbb{K}$: this pairing can be defined as
\[ \langle f,g\rangle :=-\frac{1}{2[\mathbb{K}:\mathbb{Q}]}\left(\bar{L}_{f}-\bar{L}_{g}\right)^2,\]
where $\bar{L}_{f}$ (resp. $\bar{L}_{g}$) is the adelic semi-positive continuous line bundle associated with $f$. By the arithmetic Hodge Index Theorem of Yuan and Zhang~\cite{YZ-Hodge}, we have $\langle f,g\rangle \geq0$ and $\langle  f,g\rangle=0$ if and only if $\bar{L}_{f}=\bar{L}_{g}$ (see Section~\ref{sec:energy} for more derails).

The proof of this result goes in two main steps: the first one relies on Yuan and Zhang's theory of adelic line bundles on quasi-projective varieties~\cite{YZ-adelic} and the non-vanishing of an appropriate measure induced by the pairing on the quotient space $(\mathrm{Rat}_{d_1}\times \mathrm{Rat}_{d_2})/\mathrm{PGL}(2)$. This non-vanishing property has been established for the case $d_1=d_2$ by DeMarco and Mavraki~\cite{DeMarco_Mavraki_2024} and the case when the $d_1$ and $d_2$ are multiplicatively dependent follows easily. The case when $d_1$ and $d_2$ are multiplicatively independent actually follows easily from properties of the local archimedean contribution of the Arakelov-Zhang pairing (see Section~\ref{sec:energy} for more details).

The second step is the main new ingredient we give here, and is directly inspired from an energy inequality relating the pairing of two maps $f$ and $g$ and the  height of a finite set of points for the height function $\hat{h}_f+\hat{h}_g$, when $f$ and $g$ are defined over a number field. Here, the canonical height function, as a function on the $\bar{\mathbb{Q}}$-points, is defined as
\[\hat{h}_f=\lim_{n\to\infty}\frac{1}{d^n}h_{\mathrm{nv}}\circ f^n \quad \text{on} \ \mathbb{P}^1(\bar{\mathbb{Q}}).\]
To any endomorphism $f:\mathbb{P}^1\to\mathbb{P}^1$ defined over a field $K$, we can associate a polynomial homogeneous lift $F:\mathbb{A}^{2}_K\to\mathbb{A}^{2}_K$ and two such lifts $F,F'$ necessarily satisfy $F=\alpha F'$ with $\alpha\in \bar{K}^\times$. If $(a_0,\ldots,a_{N})$ is the collection of coefficients of a lift $F$ of the map $f$, one can identify $f$ with $[a_0:\cdots:a_{2d+1}]\in \mathbb{P}^{2d+1}\setminus\{\mathrm{Res}=0\}$. 
If $f$ is defined over a number field, let
\[h_{\mathrm{Rat}_d}(f):=\sum_{v\in M_\mathbb{K}}\frac{N_v}{[\mathbb{K}:\mathbb{Q}]}\log\|F\|_v,\]
where $F$ is any lift of $f$ defined over a number field $\mathbb{K}$ and $\|F\|_v=\max\{|a_0|_v,\ldots,|a_{2d+1}|_v\}$.
By the product formula, the quantity $h_{\mathrm{Rat}_d}(f)$ is well-defined.
%
%

Inspired by \cite[Theorem~1.8]{DKY1} and \cite[Theorem~1.9]{DKY2}, we prove the following.
\begin{Theorem}\label{tm:split}
For any $d_1, d_2\geq2$, there is a constant $A\geq1$ depending only on $d_1$ and $d_2$ such that for any rational maps $(f,g)\in\mathrm{Rat}_{d_1}\times\mathrm{Rat}_{d_2}(\bar{\mathbb{Q}})$, for any $0<\delta<1$, and any finite Galois invariant subset $E\subset \mathbb{A}^1(\bar{\mathbb{Q}})$,
\[\langle f,g\rangle\leq A\left(\delta-\frac{\log(\delta)}{\# E}\right)\left(h_{\mathrm{Rat}_{d_1}}(f)+h_{\mathrm{Rat}_{d_2}}(g)+1\right)+\frac{1}{\# E}\sum_{x\in E}(\hat{h}_{f}(x)+\hat{h}_{g}(x)).\]
\end{Theorem}

The proof of this result relies on a precise analysis of the adelically H\"older regularity properties of the invariant metric for a given rational map, as well as its dependence on the parameter. We then use the approximation of measures distributed on finite Galois-invariant finite sets introduced by Favre and Rivera-Letelier~\cite{FRL} as in \cite{DKY1,DKY2} to conclude by choosing an appropriate adelic radius of approximation, see Section~\ref{sec:split}. Shang and Yap proved a variant of Theorem~\ref{tm:split} for polynomials as Theorem 3.18 in \cite{Shang-Yap}.

\subsection*{Acknowledgment}
The idea to use the H\"older regularity of the Green function of rational maps to obtain the energy estimate from Theorem~\ref{tm:split} comes from discussions with Gabriel Vigny. I would like to thank him for all the very useful discussion we had on this proof. I would also like to warmly thank the referee for very useful comments.

\section{Preliminaries}
\subsection{Basics over a metrized field}\label{sec:basics}
Let $(K,|\cdot|)$ be a complete and algebraically closed field of characteristic $0$. Let $k\geq1$ be an integer and $\pi:\mathbb{A}^{k+1}\setminus\{0\}\to\mathbb{P}^k$ be the canonical projection. For $X=(X_1,\ldots,X_{k+1})\in K^{k+1}$, we also let 
\[\|X\|:=\max_j|X_j|.\]
We also define a distance on $\mathbb{P}^k$ by letting
\[\mathrm{dist}(x,y):=\frac{\max_{1\leq i, j\leq k+1}|X_iY_j-X_jY_i|}{\|X\|\cdot\|Y\|},\]
for any $X,Y\in K^{k+1}\setminus\{0\}$ with $\pi(X)=x$ and $\pi(Y)=y$. For this distance, $\mathrm{diam}(\mathbb{P}^k)=1$.

\medskip

When $X$ is a projective variety over $K$, and $L$ is a line bundle over $X$, we denote by $X^\mathrm{an}$ (resp. $L^\mathrm{an}$) the Berkovich analytification of $X$ (resp. of $L$). A \emph{continuous metric}  on $L^\mathrm{an}$ is a the data for each local analytic section $\sigma$ defined over an open subset $U\subset X^\mathrm{an}$ of a continuous function $\|\sigma\|_U:U\to\mathbb{R}_+$ such that $\|\sigma\|_U$ vanishes only at zeroes of $\sigma$, the restriction of $\|\sigma\|_U$ to an open subset $V$ is $\|\sigma\|_V$ and for each analytic function $f$ on $U$, $\|f\sigma\|_U=|f|\times\|\sigma\|_U$.
When $(K,|\cdot|)$ is non-archimedean, we say that a metric is a \emph{model} metric if there is an $\mathscr{O}_K$-model $\mathscr{X}$ of $X$ $-$ i.e. an integral $\mathscr{O}_K$-scheme with generic fiber $\mathscr{X}_\eta$ isomorphic to $X$ and a line bundle $\mathscr{L}$ on $\mathscr{X}$ which restriction to the generic fiber isomorphic to $L$ $-$ such that the metric is on $L$ is induced by this model $(\mathscr{X},\mathscr{L})$ (see,~e.g.,~\cite[Chapter~1]{book-unlikely} for more details).

\subsection{Adelic metrizations over a number field}
For the material of this section, we refer to \cite{ACL2} and \cite{Zhang-positivity}. 
When $\mathbb{K}$ is a number field, let $M_\mathbb{K}$ be its set of places, i.e. the set of equivalence places of non-trivial norms on $\mathbb{K}$. For any $v\in M_\mathbb{K}$, denote by $|\cdot|_v$ the corresponding norm. We let $\mathbb{K}_v$ be the completion of $(\mathbb{K},|\cdot|_v)$, by $\bar{\mathbb{K}}_v$ the algebraic closure of $\mathbb{K}_v$, and finally by $\mathbb{C}_v$ the completion of $\bar{\mathbb{K}}_v$. We then denote by $\|\cdot\|_v$ the induced norm on the affine space $\mathbb{A}^k$ for any $k\geq1$ and by $\mathrm{dist}_v$ the induced distance on $\mathbb{P}^k_{\mathbb{C}_v}$.

Finally, we assume the norms $|\cdot|_v$ are normalized so that the product formula holds:
\[\prod_{v\in M_\mathbb{K}}|x|_v^{N_v}=1, \quad x\in \mathbb{K}^\times,\]
with $N_v=[\mathbb{K}_v:\mathbb{Q}_p]$, where $p$ is the residual characteristic of $(\mathbb{K},|\cdot|_v)$.

\medskip

We will use the following definitions in the whole text:
\begin{itemize}
\item  An \emph{adelic constant} (over $\mathbb{K}$) is a function $C:v\in M_\mathbb{K}\longmapsto C_v\in \mathbb{R}_+^*$ such that $C_v=1$ for all but finitely many $v\in M_\mathbb{K}$,
\item An adelic constant is \emph{small} if $0<C_v\leq1$, for all $v\in M_\mathbb{K}$.
\end{itemize}
We will denote by $C=\{C_v\}_{v\in M_\mathbb{K}}$ an adelic constant over $\mathbb{K}$.

\medskip

Let $X$ be a projective variety defined over a number field $\mathbb{K}$. Fix a place $v\in M_\mathbb{K}$. Denote by $X_v^\mathrm{an}$ the Berkovich analytification of $X$ at the place $v$.   Let $L$ be a line bundle on $X$, also defined over $\mathbb{K}$. An \emph{adelic metric} on $L$ is a collection of metrics $\{\|\cdot\|_v\}_{v\in M_\mathbb{K}}$ such that, for any $v\in M_\mathbb{K}$, $\|\cdot\|_v$ is a continuous metric on $L_v^\mathrm{an}$ and there exists an $\mathscr{O}_\mathbb{K}$-model $(\mathscr{X},\mathscr{L})$ of $(X,L)$ such that, for all but finitely many $v\in M_\mathbb{K}$, $\|\cdot\|_v$ is a model metric on $L_v^\mathrm{an}$ induced by $(\mathscr{X},\mathscr{L})$. Denote $\bar{L}:=(L,\{\|\cdot\|_v\}_{v\in M_\mathbb{K}})$. We say $\bar{L}$ is \emph{semi-positive} if for any $v\in M_\mathbb{K}$, in any local chart $U\subset X^\mathrm{an}_v$, the metric $\|\cdot\|_v$ can be written under the form $|\cdot|_ve^{-u_v}$, where $u_v$ is plurisubharmonic.
We say $\bar{L}$ is \emph{integrable} if it can be written as a difference of semi-positive adelic line bundles.
We also let $c_1(\bar{L})_v$ be the curvature form of the metric $\|\cdot\|_{v}$ on $X_v^\mathrm{an}$.

\subsection{Arithmetic intersection and heights}

Let $L_0,\ldots,L_k$ be $\mathbb{Q}$-line bundle on $X$. Assume $L_i$ is equipped with an adelic continuous metric $\{\|\cdot\|_{v,i}\}_{v\in M_\mathbb{K}}$ and we denote $\bar{L}_i:=(L_i,\{\|\cdot\|_v\}_{v\in M_\mathbb{K}})$. Assume $\bar{L}_i$ is  semi-positive for $1\leq i\leq k$ and $\bar{L}_0$ is integrable.

We will use in the sequel that the arithmetic intersection number $\left(\bar{L}_0\cdots\bar{L}_k\right)$ is symmetric and multilinear with respect to the $L_i$ and that
\[\left(\bar{L}_0\right)\cdots\left(\bar{L}_k\right)=\left(\bar{L}_1|_{\mathrm{div}(s)}\right)\cdots\left(\bar{L}_k|_{\mathrm{div}(s)}\right)+\sum_{v\in M_\mathbb{K}}N_v\int_{X_v^{\mathrm{an}}}\log\|s\|^{-1}_v \bigwedge_{j=1}^kc_1(\bar{L}_i)_v,\]
for any global section $s\in H^0(X,L_0)$ (whenever such a section exists) and wuere $N_v=[\mathbb{K}_v:\mathbb{Q}_p]$. In particular, if $L_0$ is the trivial bundle and $\|\cdot\|_{v,0}$ is the trivial metric at all places but $v_0$, this gives
\[\left(\bar{L}_0\right)\cdots\left(\bar{L}_k\right)=N_{v_0}\int_{X_{v_0}^{\mathrm{an}}}\log\|1\|^{-1}_{v_0,0} \bigwedge_{j=1}^kc_1(\bar{L}_i)_{v_0}.\]
When $\bar{L}$ is a big and nef $\mathbb{Q}$-line bundle endowed with a semi-positive continuous adelic metric, following Zhang~\cite{Zhang-positivity}, we define $h_{\bar{L}}(X)$ as
\[ h_{\bar{L}}(X):=\frac{\left(\bar{L}\right)^{k+1}}{(k+1)[\mathbb{K}:\mathbb{Q}]\mathrm{vol}(L)},\]
where $\mathrm{vol}(L)=(L)^k$ is the volume of the line bundle $L$ (also denoted by $\deg_X(L)$ sometimes).
We also define the height of a closed point $x\in X(\bar{\mathbb{Q}})$ as 
\[h_{\bar{L}}(x)=\frac{\left(\bar{L}|x\right)}{[\mathbb{K}:\mathbb{Q}]}=\frac{1}{[\mathbb{K}:\mathbb{Q}]\#\mathsf{O}(x)}\sum_{v\in M_\mathbb{K}}\sum_{\sigma:\mathbb{K}(x)\hookrightarrow \C_v}\log\|s(\sigma(x))\|_{v}^{-1},\]
where $\mathsf{O}(x)$ is the Galois orbit of $x$, for any section $s\in H^0(X,L)$ which does not vanish at $x$. Finally, for any Galois-invariant finite set $F\subset X(\bar{\mathbb{K}})$, we define $h_{\bar{L}}(F)$ as
\[h_{\bar{L}}(F):=\frac{1}{\# F}\sum_{y\in F}h_{\bar{L}}(y).\]
We now assume $X$ is a curve. The Zhang's inequalities~\cite{Zhang-positivity}
can be then written as follows:
\begin{lemma}\label{ineg-Zhang-below}
If $X$ is a smooth projective curve defined over a number field, if $L$ is ample, if $\bar{L}$ is an adelic semi-positive metrized line bundle and if we let $e(\bar{L})=\sup_Z\inf_{x\in (X\setminus Z)(\bar{\mathbb{K}})}h_{\bar{L}}(x)$, where the supremum is taken over all Zariski closed proper subsets $Z$ of $X$ defined over $\mathbb{K}$, then
\[\frac{1}{2}\left(e(\bar{L})+k\inf_{y\in X(\bar{\mathbb{K}})}h_{\bar{L}}(y)\right)\leq h_{\bar{L}}(X)\leq e(\bar{L}).\]
In particular, if $h_{\bar{L}}(x)\geq0$ for all $x\in X(\bar{\mathbb{K}})$, then $h_{\bar{L}}(X)\geq0$.
\end{lemma}
We also will use the following version of the arithmetic Hodge Index Theorem of Yuan and Zhang~\cite{YZ-Hodge}.
\begin{theorem}[Arithmetic Hodge Index]\label{tm:Hodge}
Let $X$ be a smooth projective curve defined over $\mathbb{K}$. Let $\bar{L}$ be an integrable adelic line bundle over $X$ with $L=\mathcal{O}_X$. Then
\[(\bar{L})^2\leq0.\]
\end{theorem}

When $f:\mathbb{P}^k\to\mathbb{P}^k$ is a degree $d\geq2$ endomorphism defined over a number field $\mathbb{K}$, one can define a canonical metrization on $\mathcal{O}_{\mathbb{P}^k}(1)$ as follows: let $\{\|\cdot\|_{0,v}\}_{v\in M_\mathbb{K}}$ be the naive metrization, i.e. the metrization inducing the naive height. For any $v\in M_\mathbb{K}$, if we let $\|\cdot\|_{n,v}:=((f^n)^*\|\cdot\|_{0,v})^{1/d^n}$, then the sequence
\[g_n:=\log\left( \frac{\|\cdot\|_{n,v}}{\|\cdot\|_{n-1,v}}\right)\]
is a uniform Cauchy sequence, whence $\|\cdot\|_{n,v}$ onverges uniformly to a metric $\|\cdot\|_{f,v}$ on $\mathcal{O}_{\mathbb{P}^k}(1)$ which satisfies $d^{-1}f^*\|\cdot\|_{f,v}=\|\cdot\|_{f,v}$. Moreover, if $g_f:=\sum_ng_n$, then $\|\cdot\|_{f,v}=e^{-g_f}\|\cdot\|_{0,v}$. The induced metrized line bundle $\bar{L}_f$ is adelic and semi-positive and satisfies $\hat{h}_f=h_{\bar{L}_f}$.
%

\subsection{The mutual energy pairing}\label{sec:energy}
\subsubsection{The global mutual energy}
Pick integers  $d_1,d_2\geq2$, pick rational maps $f_1\in \mathrm{Rat}_{d_1}(\bar{\mathbb{Q}})$ and $f_2\in \mathrm{Rat}_{d_2}(\bar{\mathbb{Q}})$ and let $\mathbb{K}$ be a number field such that $f_1$ and $f_2$ are both defined over $\mathbb{K}$.

Let $\bar{L}:=\frac{1}{2}(\bar{L}_{f_1}+\bar{L}_{f_2})$, where $\bar{L}_{f_i}$ is the canonical metric of $f_i$ on $\mathcal{O}_{\mathbb{P}^1}(1)$. 
\begin{lemma}\label{lm:heightpairing}
For any $x\in \mathbb{P}^1(\bar{\mathbb{Q}})$, we have $h_{\bar{L}}(x)=\frac{1}{2}(\hat{h}_{f_1}(x)+\hat{h}_{f_2}(x))$. Moreover, 
\[h_{\bar{L}}(\mathbb{P}^1)=\frac{1}{2}\langle f_1,f_2\rangle.\]
\end{lemma}
\begin{proof}
The first equality comes from directly from the definition of $\bar{L}$. Indeed, if $x\in\mathbb{P}^1(\bar{\mathbb{Q}})$, then $h_{\bar{L}}(x)=h_{\frac{1}{2}(\bar{L}_{f_1}+\bar{L}_{f_2})}(x)=\frac{1}{2}h_{\bar{L}_{f_1}}(x)+\frac{1}{2}h_{\bar{L}_{f_2}}(x)=\frac{1}{2}(\hat{h}_{f_1}(x)+\hat{h}_{f_2}(x))$.
For the second equality, we can compute
\[h_{\bar{L}}(\mathbb{P}^1)=\frac{(\bar{L})^2}{2[\mathbb{K}:\mathbb{Q}]}=\frac{(\bar{L}_{f_1}+\bar{L}_{f_2})^2}{8[\mathbb{K}:\mathbb{Q}]}=\frac{1}{8[\mathbb{K}:\mathbb{Q}]}\left((\bar{L}_{f_1})^2+(\bar{L}_{f_2})^2+2\left(\bar{L}_{f_1}\cdot\bar{L}_{f_2}\right)\right).\]
We thus have proved that $h_{\bar{L}}(\mathbb{P}^1)=\frac{1}{4}\left(h_{\bar{L}_{f_1}}(\mathbb{P}^1)+h_{\bar{L}_{f_2}}(\mathbb{P}^1)\right)+\frac{\left(\bar{L}_{f_1}\cdot\bar{L}_{f_2}\right)}{4[\mathbb{K}:\mathbb{Q}]}$.
Similarly, by definition of Zhang's pairing, we have
\[\langle f_1,f_2\rangle =-\frac{(\bar{L}_{f_1}-\bar{L}_{f_2})^2}{2[\mathbb{K}:\mathbb{Q}]}=\frac{\left(\bar{L}_{f_1}\cdot\bar{L}_{f_2}\right)}{[\mathbb{K}:\mathbb{Q}]}-\frac{1}{2}\left(h_{\bar{L}_{f_1}}(\mathbb{P}^1)+h_{\bar{L}_{f_2}}(\mathbb{P}^1)\right).\]
By Zhang's inequalities (see, e.g., Lemma~\ref{ineg-Zhang-below}), we have $\hat{h}_{f_1}(\mathbb{P}^1)=\hat{h}_{f_2}(\mathbb{P}^1)=0$. This concludes the proof.
\end{proof}

\begin{remark}\normalfont
Usually, the mutual energy is not described as above, but rather in terms of local contributions, see e.g.~\cite[Formula~(2)]{szpiro}. We will use properties of the archimedean contributions in the proof of Theorem~\ref{tm:uniform}.
\end{remark}
\subsubsection{The complex mutual energy}
The mutual energy of two signed measures $\rho,\nu$ of mass $0$ with continuous potentials on $\mathbb{P}^{1}(\mathbb{C})$ is defined in \cite{FRL} by
\[(\rho,\nu):=-\int_{\mathbb{C}^2\setminus\Delta}\log|z-w|\mathrm{d}\rho(z)\otimes \mathrm{d}\nu(w).\]
It is known that $(\rho,\rho)\geq0$ and $(\rho,\rho)=0$ if and only if $\rho=0$ (see~\cite[Propositions~2.6]{FRL}). By the classification of rational maps sharing their maximal entropy measure~\cite{Levin-Przytycki}, we get
\begin{lemma}\label{lm:samemu}
Pick two integers $d_1,d_2\geq 2$. Then 
\begin{enumerate}
\item If $d_1$ and $d_2$ are multiplicatively independent, then $(\mu_{f_1}-\mu_{f_2},\mu_{f_1}-\mu_{f_2})>0$, unless there exists $\varphi\in \mathrm{PLG}(2,\mathbb{C})$ such that both $\varphi\circ f_1\circ \varphi^{-1}$ and $\varphi\circ f_2\circ \varphi^{-1}$ are monomials, Chebychev polynomials of Latt\`es maps.
\item if $d_1$ and $d_2$ are multiplicatively dependent, the set of pairs $(f_1,f_2)\in\mathrm{Rat}_{d_1}\times\mathrm{Rat}_{d_2}(\mathbb{C})$ such that $(\mu_{f_1}-\mu_{f_2},\mu_{f_1}-\mu_{f_2})=0$ is contained in a pluripolar set.
\end{enumerate}
\end{lemma}

\begin{proof}
First remark that if $(\mu_{f_1}-\mu_{f_2},\mu_{f_1}-\mu_{f_2})=0$, then $\mu_{f_1}=\mu_{f_2}$ by \cite[Proposition~2.6]{FRL}. The lemma then follows Theorem A from~\cite{Levin-Przytycki}: it states that
\begin{enumerate}
\item[$(a)$] either there is $\varphi\in \mathrm{PGL}(2,\C)$ such that $\varphi\circ f_1\circ \varphi^{-1}$ and $\varphi\circ f_2\circ \varphi^{-1}$ are monomials, Chebychev polynomials of Latt\`es maps,
\item[$(b)$] or there are iterates $F_1$ and $F_2$ of $f_1$ and $f_2$ respectively and integers $N,M\geq1$ such that
\[(F_1^{-1}\circ F_1)\circ F_1^N = (F_2^{-1}\circ F_2)\circ F_2^M\]
where $(F_i^{-1}\circ F_i)$ is defined as a correspondence. In particular, $d_1=\deg(f_1)$ and $d_2=\deg(f_2)$ are multiplicatively dependent in this case.
\end{enumerate}
When $d_1$ and $d_2$ are multiplicatively independent, we thus are in the case $(a)$ above. This concludes the proof of the first point.

To prove the second point, we can remark that for any given $p,q,N,M\geq1$, the set of pairs $(f_1,f_2)\in \mathrm{Rat}_{d_1}\times\mathrm{Rat}_{d_2}(\mathbb{C})$ that satisfy
\[(f_1^{-p}\circ f_1^p)\circ f_1^N=(f_2^{-q}\circ f_2^q)\circ f_2^M\]
is contained in a strict closed subvariety $Z_{p,q}^{N,M}$ of $\mathrm{Rat}_{d_1}\times\mathrm{Rat}_{d_2}(\mathbb{C})$. Whence we can write
\[\{(f_1,f_2)\in \mathrm{Rat}_{d_1}\times\mathrm{Rat}_{d_2}(\mathbb{C})\, : \ (\mu_{f_1}-\mu_{f_2},\mu_{f_1}-\mu_{f_2})=0\}\subset \bigcup_{p,q,N,M\geq1}Z_{p,q}^{N,M}\cup E,\]
where $E$ is the union of the orbits of pairs of monomial maps, pairs of Chebychev polynomials and pairs of Latt\`es maps, under the action by simultaneous conjugacy. In particular $E$ is a proper closed subvariety of $\mathrm{Rat}_{d_1}\times\mathrm{Rat}_{d_2}(\mathbb{C})$ and $\bigcup_{p,q,N,M\geq1}Z_{p,q}^{N,M}\cup E$ is pluripolar.
\end{proof}

\section{H\"older estimates for the dynamical Green functions}\label{sec:Holder}
In this section, we let $(K,|\cdot|)$ be a complete and algebraically closed field of characteristic $0$ and $k\geq1$ be an integer. Our main aim here is to show that, when varying in an algebraic family, the Green function of an endomorphism of $\mathbb{P}^k_K$ has explicitly controlled H\"older constant and exponent with the parameter.

\medskip

Pick a degree $d$ endomorphism $f:\mathbb{P}^k_{K}\to\mathbb{P}^k_{K}$ defined over $K$ and let $F:\mathbb{A}^{k+1}_{K}\to\mathbb{A}^{k+1}_{K}$ be a homogeneous polynomial lift of $f$. For any $z\in\mathbb{P}^k(K)$ and any $x\in K^{k+1}\setminus\{0\}$ with $\pi(x)=z$, we let
\[u_F(z):=\frac{1}{d}\log\|F(x)\|-\log\|x\|.\]
Note that it depends on a choice of lift $F$ we have $u_{\alpha F}=u_{F}+\frac{1}{d}\log|\alpha|$ for all $\alpha\in K^*$. We also define the \emph{Green function} of $F$ as \[g_F:=\sum_{j\geq0}d^{-j}u_F\circ f^j.\]
It also depends on a choice of lift $F$: we have $g_{\alpha F}=g_F+\frac{1}{d-1}\log|\alpha|$ for any $\alpha\in K^*$.

\medskip

As the map $F$ is a polynomial map, it can be identified with a point in $\mathbb{A}_K^{N+1}$ where $N+1=(k+1)\frac{(d+k)!}{k!d!}$ via its coefficients. We let
\[\|F\|:=\max\{|a|\, : \ a \ \text{coefficient of} \ F\}.\]

This section is devoted to the proof of the following.

\begin{theorem}\label{tm-Green-Holder-adelic}
The Green function $g_F$ is H\"older continuous. More precisely, there are constants $C_1,C_2\geq1$ such that $C_1=C_2=1$ if $K$ is non-archimedean and $C_1$ and $C_2$ depend only on $d$ and $k$ if $K$ is archimedean and a constant $C_3\geq1$ depending only on $k$ and $d$ such that
\[|g_F(z)-g_F(w)|\leq C_3\left(C_1+\log^+\|F\|+\log^+\frac{\|F\|^{(k+1)d^k}}{|\mathrm{Res}(F)|}\right)\cdot\mathrm{dist}(z,w)^{\alpha},\]
for all $z,w\in\mathbb{P}^k(K)$, where, if $(K,|\cdot|)$ is archimedean
\[\alpha:=\displaystyle\log d/\left(C_2+\log\max\left\{2d,\frac{\|F\|^{(k+1)d^k}}{|\mathrm{Res}(F)|}\right\}\right),\]
and if $(K,|\cdot|)$ is non-archimedean, $\alpha$ satisfies
\[\alpha=\left\{\begin{array}{ll}
1 & \text{if} \ \|F\|^{(k+1)d^k}=|\mathrm{Res}(F)| \ \text{and},\\
\displaystyle\log d/2\log\left(\frac{\|F\|^{(k+1)d^k}}{|\mathrm{Res}(F)|}\right)  & \text{otherwise}.
\end{array}\right.\]
\end{theorem}

\subsection{Estimates for building blocks of the Green function}
From now on, unless specified, $K$ is an algebraically closed field of characteristic zero which is complete with respect to a non-trivial norm $|\cdot|$. The following is a consequence of basic properties of McCaulay resultant.
\begin{lemma}\label{lm:supuF}
There is a constant $C_1(K,d,k)\geq1$ such that $C_1(K,d,k)=1$ if $K$ is non-archimedean and $C_1(K,d,k)$ depends only on $d$ and $k$ when $K$ is archimedean, and such that for any $F$ as above, 
\[\sup_{z\in\mathbb{P}^k(K)}|u_F(z)|\leq \log\max\{\|F\|,1\}+\log C_1(K,d,k).\]
\end{lemma}
\begin{proof}
An application of the MacCaulay resultant gives a constant $C(K)\geq1$ which is $1$ whenever $K$ is non-archimedean and depending only on $d$ and $k$ when $K$ is archimedean, such that
\begin{align}
\frac{|\mathrm{Res}(F)|}{C(K)\|F\|^{(k+1)d^k-1}}\|x\|^d\leq \|F(x)\|\leq C(K)\|F\|\cdot\|x\|^d, \quad x\in K^{k+1}\setminus\{0\}.\label{resultant}
\end{align}
This rewrites, for $z\in\mathbb{P}^k(K)$ and $x\in K^{k+1}$ with $\pi(x)=z$, as
\begin{align*}
|u_F(z)| & =\left|\frac{1}{d}\log\frac{\|F(x)\|}{\|x\|^d}\right|\\
 & \leq  \max\left\{0,\log\|F\|+\log C(K),\log\frac{|\mathrm{Res}(F)|}{\|F\|^{(k+1)d^k-1}}+\log C(K)\right\}.
\end{align*}
As the resultant map $\mathrm{Res}:\mathbb{A}^{N+1}\to \mathbb{A}^1$ is homogeneous with degree $(k+1)d^k$ and is defined over $\mathbb{Z}$, we have
\[\log\frac{|\mathrm{Res}(F)|}{\|F\|^{(k+1)d^k-1}}\leq \log\|F\|+C'(K),\]
where $C'(K)\geq0$ and $C'(K)=0$ if $K$ is non-archimedean and depending only on $d$ and $k$ when $K$ is archimedean. This concludes the proof. 
\end{proof}

Following the strategy of \cite[Theorem~13]{Kawaguchi-Silverman-Green}, we control the lipschitz constant of $u_F$.
\begin{lemma}\label{lm:lipuF}
There is a constant $C_2(K,d,k)\geq1$ such that $C_2(K,d,k)=1$ if $K$ is non-archimedean and $C_2(K,d,k)$ depends only on $d$ and $k$ when $K$ is archimedean, and such that for any $F$ as above, 
\[|u_F(z)-u_F(w)|\leq \frac{C_2(K,d,k)^2\|F\|^{(k+1)d^k}}{d|\mathrm{Res}(F)|}\cdot 
\log\left(\frac{C_2(K,d,k)^2\|F\|^{(k+1)d^k}}{|\mathrm{Res}(F)|}\right)\cdot\mathrm{dist}(z,w).\]
Moreover, if $(K,|\cdot|)$ is non-archimedean and $\mathrm{dist}(x,y)\leq|\mathrm{Res}(F)|/\|F\|^{(k+1)d^k}$, then $u_F(x)=u_F(y)$.
\end{lemma}

\begin{proof}
Let $x,y\in \mathbb{P}^k(K)$ and $X,Y\in K^{k+1}$ with $\pi(X)=x$ and $\pi(Y)=y$. First, if 
\[\mathrm{dist}(x,y)\geq R:=\frac{|\mathrm{Res}(F)|}{C(K)^2\|F\|^{(k+1)d^k}},\]
the formula \eqref{resultant} gives
\begin{align*}
|u_F(x)-u_F(y)| & =\left|\frac{1}{d}\log\frac{\|F(X)\|\|Y\|^d}{\|F(Y)\|\|X\|^d}\right|\leq \frac{1}{d}\log \left(\frac{1}{R}\right)\\
& \leq \frac{1}{d}\frac{1}{R}\log \left(\frac{1}{R}\right)\cdot\mathrm{dist}(x,y).
\end{align*}
We now assume $\mathrm{dist}(x,y)<R\leq 1$.  Pick $X,Y\in K^{k+1}$ such that $\pi(X)=x$ and $\pi(Y)=y$ with $\|X\|=\|Y\|=1$. If $K$ is non-archimedean, following the proof of \cite[Theorem~13]{}, but we reproduce the argument here for the sake of completeness. By the strong triangle inequality, we can assume $|x_k|=|y_k|=1$. Write 
\[F(X+h)=F(X)+\sum_{j=0}^kh_iB_i(X,h),\]
where $B_i\in K[X,h]$ and its coefficients are linear combinations of coefficients of $F$. Since $|y_k|=1$ and $F$ is homogeneous, we can compute
\begin{align*}
\|F(X)\| & =\|F(x_kY+y_kX-x_kY)\|\\
& =\|F(x_kY)+\sum_{j=0}^k(y_kx_j-x_ky_j)B_j(x_yY,y_kX-x_kY)\|.
\end{align*}
By our assumption, we have $R\leq \frac{\|F(X)\|}{\|F(Y)\|}\leq R^{-1}$ and
\[\left\|\sum_{j=0}^k(y_kx_j-x_ky_j)B_j(x_yY,y_kX-x_kY)\right\|<R\|F(X)\| \leq\|F(x_kY)\|=\|F(Y)\|.\]
This implies $\|F(X)\|=\|F(Y)\|$ and
\begin{align*}
u_F(x)-u_F(y)=\frac{1}{d}\log\frac{\|F(X)\|}{\|F(Y)\|} & =0,
\end{align*}
and the proof is complete in this case.

If $K$ is archimedean, we can assume there is $j$ such that $y_j=1$ and there is $C\geq1$ depending only on $d$ and $k$ such that
\begin{align*}
|x_j|^d\frac{\|F(Y)\|}{\|F(X)\|} & =\frac{\|F(x_jY)\|}{\|F(X)\|}=\frac{\left\|F(X)+D_XF(x_jY-y_jX)\right\|}{\|F(X)\|}\\
& \leq 1+\frac{\|D_XF(x_jY-y_jX)\|}{\|F(X)\|}\leq 1+C \|F\|\frac{\max_\ell|x_jy_\ell-y_jx_\ell|}{\|F(X)\|}\\
& \leq 1+C\frac{1}{R}\mathrm{dist}(x,y).
\end{align*}
Up to increasing $C$ and up to replacing $C(K)^2$ by $C\cdot C(K)^2$ and $R$ by $R/C$, we have $\log(R)\leq -1$, whence, taking the supremum over $j$ gives
\begin{align*}
\frac{\|F(Y)\|}{\|F(X)\|} \leq 1+\frac{1}{R}\log\left(\frac{1}{R}\right)\mathrm{dist}(x,y).
\end{align*}
We now use that $log(1+t)\leq t$ for all $t\geq0$ and get
\[u_F(x)-u_F(y)\leq \frac{1}{d}\frac{1}{R}\log\left(\frac{1}{R}\right)\mathrm{dist}(x,y).\]
We conclude reversing the roles of $x$ and $y$.
\end{proof}

Finally, we control the lipschitz constant of $F$, relying on McCaulay resultant.
\begin{lemma}\label{lm:lipf}
There is a constant $C_3(K,d,k)\geq1$ such that $C_3(K,d,k)=1$ if $K$ is non-archimedean and $C_3(K,d,k)$ depends only on $d$ and $k$ when $K$ is archimedean, and such that for any $F$ as above, and any $x,y\in \mathbb{P}^k(K)$,
\[\mathrm{dist}(f(x),f(y)) \leq C_3(K,d,k)\left(\frac{\|F\|^{(k+1)d^k}}{|\mathrm{Res}(F)|}\right)^2\cdot\mathrm{dist}(x,y).\]
\end{lemma}
\begin{proof}
Pick $x,y\in \mathbb{P}^k(K)$. For $X,Y\in K^{k+1}\setminus\{0\}$, if $x=\pi(X)$ and $y=\pi(Y)$, we have
\begin{align*}
\frac{\mathrm{dist}(f(x),f(y))}{\mathrm{dist}(x,y)} & =\frac{\|X\|\cdot\|Y\|\cdot\max_{1\leq i, j\leq k+1}|F_i(X)F_j(Y)-F_i(Y)F_j(X)|}{\|F(X)\|\cdot\|F(Y)\|\cdot\max_{1\leq i, j\leq k+1}|X_iY_j-X_jY_i|}\\
& \leq C(K)\cdot\frac{\|F\|^{2(k+1)d^k-2}}{|\mathrm{Res}(F)|^2}\cdot \frac{\max_{1\leq i, j\leq k+1}|F_i(X)F_j(Y)-F_i(Y)F_j(X)|}{(\|X\|\cdot\|Y\|)^{d-1}\cdot\max_{1\leq i, j\leq k+1}|X_iY_j-X_jY_i|}.
\end{align*}
We now use that for any $i,j$, there is $\ell,m$ such that $X_mY_\ell-X_\ell Y_m$ divides $F_i(X)F_j(Y)-F_i(Y)F_j(X)$ as a polynomial in the variables $X_1,\ldots,X_{k+1},Y_1,\ldots,Y_{k+1}$. Since $F_j$ and $F_i$ are homogeneous of degree $d$, for any $i,j$ we have
\[|F_i(X)F_j(Y)-F_i(Y)F_j(X)|\leq C'(K)\|F\|^2\cdot (\|X\|\cdot\|Y\|)^{d-1}\cdot\max_{m,\ell}|X_mY_\ell-X_\ell Y_m|,\]
where $C'(K)\geq1$ depends only on $d$ and $k$ if $K$ is archimedean and $C'(K)=1$ if $K$ is non-archimedean. This gives the wanted inequality.
\end{proof}

%

\subsection{H\"older regularity of the Green function: proof of Theorem~\ref{tm-Green-Holder-adelic}}

Recall that we defined the Green function of the lift $F$ as $g_F:=\sum_{j\geq0}d^{-j}u_F\circ f^j$.
%
%
The proof consists in first giving a H\"older estimate involving the next three constants depending on $F$, and then estimating those constants. Define
\begin{align*}
\left\{\begin{array}{lll}
\displaystyle C_1(F) & := & \sup_{z\in\mathbb{P}^k(K)}|u_F(z)|,\\
\displaystyle C_2(F) & := & \inf\{L\geq0\, : \ \forall z,w\in\mathbb{P}^k(K), \ |u_F(z)-u_f(w)|\leq L \cdot\mathrm{dist}(z,w)\},\\
C_3(F) & := & \inf\{L\geq0\, : \ \forall z,w\in\mathbb{P}^k(K), \ \mathrm{dist}(f(z),f(w))\leq L \cdot\mathrm{dist}(z,w)\}.
\end{array}
\right.
\end{align*}
We follow a classical argument from complex dynamics (see,~e.g.~\cite{dinhsibony2}): pick $z,w\in\mathbb{P}^k(K)$ and  fix an integer $N\geq1$. First assume $(K,|\cdot|)$ is archimedean. By definition of $g_F$,
\begin{align*}
|g_F(z)-g_F(w)| & \leq \sum_{n\geq0}\frac{1}{d^n}|u_F(f^n(z))-u_F(f^n(w))|\\
& \leq \sum_{n=0}^{N-1}\frac{1}{d^n}|u_F(f^n(z))-u_F(f^n(w))| +2C_1(F)\sum_{n\geq N}d^{-n}\\
& \leq C_2(F)\sum_{n=0}^{N-1}\frac{1}{d^n}\mathrm{dist}
(f^n(z),f^n(w)) +2C_1(F)\frac{d^{-N}}{d-1}\\
& \leq C_2(F)\sum_{n=0}^{N-1}\left(\frac{C_3(F)}{d}\right)^n\mathrm{dist}(z,w) +2C_1(F)\frac{d^{-N}}{d-1}.
\end{align*}
We now use Lemmas~\ref{lm:supuF},~\ref{lm:lipuF} and~\ref{lm:lipf}. We have
 \begin{enumerate}
 \item $C_1(F)\leq  \log\max\{\|F\|,1\}+\log C_4$,
 \item $C_2(F)\leq C_4^2\displaystyle\frac{\|F\|^{(k+1)d^k}}{d|\mathrm{Res}(F)|} 
\log\left(\frac{C_4^2\|F\|^{(k+1)d^k}}{|\mathrm{Res}(F)|}\right)$, and
 \item $C_3(F)\leq C_4\displaystyle \left(\frac{\|F\|^{(k+1)d^k}}{|\mathrm{Res}(F)|}\right)^2$,
 \end{enumerate}
 where $C_4=C_4(K,d,k):=\max\{C_1(K,d,k),C_2(K,d,k),C_3(K,d,k)\}\geq1$ is a constant depending only on $d$ and $k$ when $K$ is archimedean and $C_4=1$ when $K$ is non-archimedean.
We first treat the case where $(K,|\cdot|)$ is archimedean. We then can replace $C_3(F)$ by $\max\{2d,C_3(F)\}\geq d+2$ and the above gives
\[|g_F(z)-g_F(w)| \leq C_2(F)\left(\frac{\max\{2d,C_3(F)\}}{d}\right)^{N}\mathrm{dist}(z,w) +C_1(F)\cdot d^{-N}.\]
This in particular implies 
\begin{align*}
C_2(F)\leq \ & \frac{C_4^2}{d}\displaystyle\max\left\{2d,C_4\left(\frac{\|F\|^{(k+1)d^k}}{|\mathrm{Res}(F)|}\right)^2\right\} \\
 & \ \ \ \times \log\left(C_4\frac{\|F\|^{(k+1)d^k}}{|\mathrm{Res}(F)|}\right),
\end{align*}
so that the above rewrites as
 \begin{align*}
 |g_F(z)-g_F(w)|\leq \  2C_4^2 &\log\max\left(C_4\frac{\|F\|^{(k+1)d^k}}{|\mathrm{Res}(F)|}\right)\\
 & \times \left(\frac{C_4}{d}\cdot\max\left\{2d,\frac{\|F\|^{2(k+1)d^k}}{|\mathrm{Res}(F)|^2}\right\}\right)^{N+1}\\ 
& \ \ \ \ \times \mathrm{dist}(z,w) + \frac{\log^+\|F\|+\log C_4}{d^N}.
 \end{align*}
Recall that $\mathrm{dist}(z,w)\leq 1$ by definition.
We now choose $N\geq0$ so that
\[N+1\geq -\log \mathrm{dist}(z,w)/\log\left(C_4\max\left\{2d,\frac{\|F\|^{2(k+1)d^k}}{|\mathrm{Res}(F)|^2}\right\}\right)\geq N.\] 
Then $d^{-(N+1)}\leq \mathrm{dist}(z,w)^{\log d/\log\left(C_4\max\left\{2d,\frac{\|F\|^{2(k+1)d^k}}{|\mathrm{Res}(F)|^2}\right\}\right)}$ and the above gives
 \begin{align*}
 |g_F(z)- & g_F(w)|\leq  C_4^2 \left(2\log\left(\frac{\|F\|^{(k+1)d^k}}{|\mathrm{Res}(F)|}\right)+2d\log C_4^2+d\log^+\|F\|\right) \times \mathrm{dist}(z,w)^\alpha,
 \end{align*}
 where $\alpha=\log d/\left(\log C_4+\log\max\left\{2d,\frac{\|F\|^{2(k+1)d^k}}{|\mathrm{Res}(F)|^2}\right\}\right)$, ending the proof in this case.
 
 \bigskip
 
\par\noindent We now assume $(K,|\cdot|)$ is non-archimedean. Then we have
 \begin{enumerate}
 \item $C_1(F)\leq  \log ^+\|F\|$,
 \item $C_2(F)\leq \displaystyle\frac{\|F\|^{(k+1)d^k}}{d|\mathrm{Res}(F)|} 
\log\left(\frac{\|F\|^{(k+1)d^k}}{|\mathrm{Res}(F)|}\right)$, and
 \item $C_3(F)\leq\displaystyle \left(\frac{\|F\|^{(k+1)d^k}}{|\mathrm{Res}(F)|}\right)^2$,
 \end{enumerate}
Let $R:=|\mathrm{Res}(F)|/\|F\|^{(k+1)d^k}\leq 1$. By Lemma~\ref{lm:lipuF} and by definition of $g_F$, if $R=1$, then $g_F\equiv0$ and we can set $\alpha:=1$.
We thus assume $R<1$ and pick $x,y\in\mathbb{P}^k(K)$. If $\mathrm{dist}(f^{N}(x),f^{N}(y))\leq R^2\leq R\leq 1$ for any $N\geq0$, we have $g_F(x)=g_F$ by construction and Lemma~\ref{lm:lipuF}. In particular, this applies if $x=y$. Otherwise, there is $N$ minimal such that $\mathrm{dist}(f^{N}(x),f^{N}(y))>R$. As $C_3(F)\leq R^{-2}$ and $N$ is minimal, this implies 
$R^2\leq R^{-2N}\mathrm{dist}(x,y)$ and $R^2>R^{-2(N-1)}\mathrm{dist}(x,y)$, i.e. 
\[R^{2(N+1)}\leq \mathrm{dist}(x,y)\leq R^{2N}.\]
By Lemma~\ref{lm:lipuF}, we deduce that
\begin{align*}
|g_F(z)-g_F(w)| & \leq\sum_{n\geq0}\frac{1}{d^n}|u_F(f^n(z))-u_F(f^n(w))|\\
& \leq\sum_{n\geq N}\frac{1}{d^n}|u_F(f^n(z))-u_F(f^n(w))|\\
& \leq 2C_1(F)\sum_{n\geq N}d^{-n}\leq 2C_1(F)\frac{d^{-N}}{d-1}.
\end{align*}
By definition of $N$, one has $2N\leq \frac{\log \mathrm{dist}(x,y)}{\log R}\leq 2(N+1)$ and the above gives
\[|g_F(x)-g_F(y)|\leq 2C_1(F)\frac{d^{-\frac{\log\mathrm{dist}(x,y)}{2\log R}}}{d-1}=2C_1(F)\frac{\mathrm{dist}(x,y)^{-\frac{\log d}{2\log R}}}{d-1}.\]
By definition of $R$, this concludes the proof.

\section{H\"older estimates and an inequality for the mutual energy}

\subsection{A height estimate for H\"older adelic line bundle}
As above, we use the notation $L=\mathcal{O}_{\mathbb{P}^1}(1)$.
We say $\bar{L}=(L,\{\|\cdot\|_v\}_{v\in M_\mathbb{K}})$ is an \emph{adelically H\"older line bundle} if $\bar{L}$ is semi-positive continuous line bundle in the sense of Zhang and if there exists 
\begin{itemize}
\item an adelic constant $C:=\{C_v\}_{v\in M_\mathbb{K}}$ with $C_v\geq e$ at archimedean places, and
\item an adelic constant $\alpha:=\{\alpha_v\}_{v\in M_\mathbb{K}}$ which is small at archimedean places,
\end{itemize}
such that for any $v\in M_\mathbb{K}$, if $g_v:=\log\|\cdot\|_v/\|\cdot\|_{v,0}$,then any $z,w\in \mathbb{P}^1(\mathbb{C}_v)$, we have
\[|g_v(z)-g_v(w)|\leq \log(C_v) \mathrm{dist}_v(z,w)^{\alpha_v}.\]
Here $\mathrm{dist}_v$ denotes the chordal distance on $\mathbb{P}_{\mathbb{C}_v}^1$ defined in \S\ref{sec:basics}.

\medskip

We prove here the following.

\begin{theorem}\label{tm:holder}
 Assume $\bar{L}$ is an adelically H\"older metrized line bundle with adelic constants $C=\{C_v\}_{v\in M_\mathbb{K}}$ and $\alpha=\{\alpha_v\}_{v\in M_\mathbb{K}}$. Let $\varepsilon=\{\varepsilon_v\}_{v\in M_\mathbb{K}}$ be a small adelic constant and let $E\subset \mathbb{A}^1(\bar{\mathbb{K}})$ be a $\mathrm{Gal}(\bar{\mathbb{K}}/\mathbb{K})$-invariant set. Then
\begin{align*}
h_{\bar{L}}(\mathbb{P}^1)\leq & \sum_{v\in M_\mathbb{K}}\frac{N_v}{[\mathbb{K}:\mathbb{Q}]}\left(2\log(C_v)\varepsilon_v^{\alpha_v}-\frac{\log(\varepsilon_v)}{2\# E}\right)+\frac{1}{\# E}\cdot\sum_{x\in E}h_{\bar{L}}(x).
\end{align*}
\end{theorem}

\begin{proof}
Let $s\in H^0(\mathbb{P}^1,L)$ be non-zero with $\mathrm{div}(s)=[\infty]$. We let $\bar{L}_0:=(L,\{\|\cdot\|_{v,0}\}_{v\in M_\mathbb{K}})$ be the standard (naive) adelic metrization on $L$. 

For a given a finite $\mathrm{Gal}(\bar{\mathbb{K}}/\mathbb{K})-$invariant set $E\subset \mathbb{A}^1(\bar{\mathbb{K}})$ and a given small adelic constant $\varepsilon:=\{\varepsilon_v\}_{v\in M_{\mathbb{K}}}$, define a metrization $\{\|\cdot\|_{\varepsilon_v,v}\}_{v\in M_\mathbb{K}}$ on $L$ as follows: on $\mathbb{A}^1_{\mathbb{C}_v}$, let 
\[g_{\varepsilon,v}(z):=\frac{1}{\# E}\sum_{x\in E}\log\max\{|z-x|_v,\varepsilon_v\}, \quad z\in\mathbb{A}^1(\mathbb{C}_v).\]
The function $g_{\varepsilon,v}$ satisfies $g_{\varepsilon,v}(z)=\log|z|_v$ for all $z\in \mathbb{A}^1(\mathbb{C}_v)$ with $|z|_v$ large enough. In particular,
if for for any degree one polynomial function $\sigma$, we set
\[\|\sigma(z)\|_{\varepsilon,v}:=|\sigma(z)|_ve^{-g_{\varepsilon,v}(z)}, \quad z\in\mathbb{A}^1(\mathbb{C}_v),\]
we define a metric $\|\cdot\|_{\varepsilon,v}$ on $L$ which is continuous. Denote by $\bar{L}_\varepsilon$ the adelic line bundle $\bar{L}_\varepsilon:=(L,\{\|\cdot\|_{\varepsilon,v}\}_{v\in M_\mathbb{K}})$.
The adelic line bundle $\bar{L}-\bar{L}_\varepsilon$ is an integrable line bundle which underlying line bundle is the trivial bundle $\mathcal{O}_{\mathbb{P}^1}$. We thus can apply the Arithmetic Hodge Index Theorem of Yuan and Zhang (see Theorem~\ref{tm:Hodge}): we have
\begin{align}
(\bar{L})^2-2(\bar{L}\cdot\bar{L}_\varepsilon)+(\bar{L}_\varepsilon)^2=(\bar{L}-\bar{L}_\varepsilon)^2\leq0.\label{ineg-Hodge-0}
\end{align}
We can easily deduce the following from \cite{FRL} and \cite{Fili}:
\begin{lemma}\label{lm:energy}
We have $(\bar{L}_\varepsilon)^2\geq \displaystyle\frac{1}{\# E}\sum_{v\in M_\mathbb{K}}N_v\log(\varepsilon_v)$.
\end{lemma}
Take Lemma~\ref{lm:energy} for granted. Combining \eqref{ineg-Hodge-0} with Lemma~\ref{lm:energy}, we find
\begin{align}
(\bar{L})^2\leq 2(\bar{L}\cdot\bar{L}_\varepsilon)-(\bar{L}_\varepsilon)^2\leq 2(\bar{L}\cdot\bar{L}_\varepsilon)-\frac{1}{\# E}\sum_{v\in M_\mathbb{K}}N_v\log(\varepsilon_v).\label{ineg:Arithmetic-Hodge}
\end{align}
To conclude the proof, we need to estimate $(\bar{L}\cdot\bar{L}_\varepsilon)$ from above.
Up to replacing $\mathbb{K}$ by an extension $\mathbb{L}$ and to normalizing the computations by $1/[\mathbb{L}:\mathbb{K}]$, we can assume $E\subset \mathbb{A}^1(\mathbb{K})$.
As above, let $s$ be the section with $\mathrm{div}(s)=[\infty]$. Then 
\[(\bar{L}\cdot\bar{L}_\varepsilon)=(\bar{L}_\varepsilon|\mathrm{div}(s))+\sum_{v\in M_\mathbb{K}}N_v\int_{\mathbb{A}_v^{1,\mathrm{an}}}\log\|s\|_v^{-1}c_1(\bar{L}_\varepsilon)_v.\]
By construction, $(\bar{L}_\varepsilon|\mathrm{div}(s))=(\bar{L}_0|\mathrm{div}(s))=0$ and, as $s$ does not vanish on $E$, we have
\[\frac{[\mathbb{K}:\mathbb{Q}]}{\# E}\sum_{x\in E}h_{\bar{L}}(x)=\frac{1}{\# E}\sum_{x\in E}\sum_{v\in M_\mathbb{K}}N_v\log\|s(x)\|_v^{-1}.\]
By definition of $\bar{L}_\varepsilon$,  for any $v\in M_\mathbb{K}$ one can write 
$c_1(\bar{L}_\varepsilon)_v=\frac{1}{\# E}\sum_{x\in E}\delta_{\zeta_{x,\varepsilon_v}}$ as measures on $\mathbb{A}^{1,\mathrm{an}}_v$. In particular, if we let $\zeta_{x,\varepsilon_v}$ be point of $\mathbb{A}^{1,\mathrm{an}}_v$ corresponding to the seminorm $\sup_{\bar{B}(x,\varepsilon_v)}|\cdot|$ on $\C_v[T]$, we find
\begin{align*}
\int_{\mathbb{A}_v^{1,\mathrm{an}}}\log\|s\|_v^{-1}c_1(\bar{L}_\varepsilon)_v & =\frac{1}{\# E}\sum_{x\in E}\int_{\mathbb{A}_v^{1,\mathrm{an}}}\log\|s\|_v^{-1}\delta_{\zeta_{x,\varepsilon_v}}\\
& = \frac{1}{\# E}\sum_{x\in E}\int_{\mathbb{A}_v^{1,\mathrm{an}}}\log\left(\frac{\|s(x)\|_v}{\|s\|_v}\right)\delta_{\zeta_{x,\varepsilon_v}} +\frac{1}{\# E}\sum_{x\in E}\log\|s(x)\|_v^{-1}.
\end{align*}
Denote as above $g_v=\log\|\cdot\|_v/\|\cdot\|_{v,0}$. The function $g_v$ satisfies 
\[|g_v(z)-g_v(w)|\leq \log(C_v)\mathrm{dist}_v(z,w)^{\alpha_v}\leq \log(C_v)\|z-w\|_v^{\alpha_v}.\]
for all $z,w\in \mathbb{A}^{1}(\mathbb{C}_v)$. The above gives
\begin{align*}
\int_{\mathbb{A}_v^{1,\mathrm{an}}}\log\|s\|_v^{-1}c_1(\bar{L}_\varepsilon)_v & = \frac{1}{\# E}\sum_{x\in E}\int_{\mathbb{A}_v^{1,\mathrm{an}}}(g_v(x)-g_v)\delta_{\zeta_{x,\varepsilon_v}}\\
& \quad +\frac{1}{\# E}\sum_{x\in E}\int_{\mathbb{A}_v^{1,\mathrm{an}}}\log\left(\frac{\|s(x)\|_{v,0}}{\|s\|_{v,0}}\right)\delta_{\zeta_{x,\varepsilon_v}}\\
& \quad \quad +\frac{1}{\# E}\sum_{x\in E}\log\|s(x)\|_v^{-1}.
\end{align*}
By assumption, we deduce
\begin{align*}
\int_{\mathbb{A}_v^{1,\mathrm{an}}}\log\|s\|_v^{-1}c_1(\bar{L}_\varepsilon)_v & \leq \log(C_v)\varepsilon_v^{\alpha_v}+\frac{[\mathbb{K}:\mathbb{Q}]}{\# E}\sum_{x\in E}\log\|s(x)\|_v^{-1}\\
& \quad +\frac{1}{\# E}\sum_{x\in E}\int_{\mathbb{A}_v^{1,\mathrm{an}}}\log\left(\frac{\|s(x)\|_{v,0}}{\|s\|_{v,0}}\right)\delta_{\zeta_{x,\varepsilon_v}}.
\end{align*}
An easy computation gives the next lemma
\begin{lemma}\label{lm:facile}
$\displaystyle\int_{\mathbb{A}_v^{1,\mathrm{an}}}\log\left(\frac{\|s(x)\|_{v,0}}{\|s\|_{v,0}}\right)\delta_{\zeta_{x,\varepsilon_v}}
 \leq \left\{\begin{array}{ll}
\varepsilon & \text{if} \ v \ \text{is archimedean},\\
0 & \text{if} \ v \ \text{is non-archimedean},
\end{array}\right.$
\end{lemma}

For $v$ archimedean, since $C_v\geq e$ and $\alpha_v\leq 1$, we have $\varepsilon_v\leq \log(C_v)\varepsilon_v^{\alpha_v}$. Pick again $v\in M_\mathbb{K}$. 
This implies
\begin{align}
\int_{\mathbb{A}_v^{1,\mathrm{an}}}\log\|s\|_v^{-1}c_1(\bar{L}_\varepsilon)_v & \leq 2\log(C_v)\varepsilon_v^{\alpha_v}+\frac{[\mathbb{K}:\mathbb{Q}]}{\# E}\sum_{x\in E}\log\|s(x)\|_v^{-1}.\label{ineg:place-by-place}
\end{align}
Since $h_{\bar{L}}(\mathbb{P}^1)=\frac{(\bar{L})^2}{2[\mathbb{K}:\mathbb{Q}]}$, the conclusion follows from summing \eqref{ineg:place-by-place} over all places $v\in M_\mathbb{K}$ and using \eqref{ineg:Arithmetic-Hodge}.
\end{proof}

\begin{proof}[Proof of Lemma~\ref{lm:energy}]
By definition, we have
\[(\bar{L}_\varepsilon^2)=(\bar{L}_\varepsilon| \mathrm{div}(s))+\sum_{v\in M_\mathbb{K}}N_v\int_{\mathbb{A}^{1,\mathrm{an}}_v}\log\|\cdot\|^{-1}_{\varepsilon,v}c_1(\bar{L}_\varepsilon)_v,\]
where $s$ is a section of $L$ with divisor $[\infty]$.
By construction of $\bar{L}_\varepsilon$, we have $(\bar{L}_\varepsilon|C\cap H_\infty)=(\bar{L}_0|C\cap H_\infty)=0$. Pick now $v\in M_\mathbb{K}$, then $c_1(\bar{L}_\varepsilon)_v=dd^c(g_{\varepsilon,v})$, so that
\begin{align*}
(\bar{L}_\varepsilon^2) & =(\bar{L}_0|C\cap H_\infty)+\sum_{v\in M_\mathbb{K}}N_v\int_{C^{\mathrm{an}}_v}g_{\varepsilon,v}dd^cg_{\varepsilon,v}.
\end{align*}
Fix $v\in M_\mathbb{K}$. Lemma~12 from \cite{Fili} and lemma~4.11 from \cite{FRL} rewrite as
\begin{align*}
\int_{C^{\mathrm{an}}_v}g_{\varepsilon,v}dd^cg_{\varepsilon,v}\geq \frac{1}{\# E}\log(\varepsilon_v)+\frac{1}{(\# E)^2}\sum_{x\neq y\in E}\log|x-y|_v
\end{align*}
and the product formula implies 
\[\sum_{v\in M_\mathbb{K}}N_v\log|x-y|_v=0.\]
This concludes the proof.
\end{proof}

\begin{proof}[Proof of Lemma~\ref{lm:facile}]
By the triangle inequality,
\begin{align*}
\int_{\mathbb{A}_v^{1,\mathrm{an}}}\log\left(\frac{\|s(x)\|_{v,0}}{\|s\|_{v,0}}\right)\delta_{\zeta_{x,\varepsilon_v}}
& \leq \left\{\begin{array}{ll}
\log^+\left(\|z-x\|_v+\|x\|_v\right)-\log^+\|x\|_v & \text{if} \ v \ \text{archimedean},\\
\log^+\max\{\|z-x\|_v,\|x\|_v\}-\log^+\|x\|_v & \text{otherwise},
\end{array}\right.
\end{align*}
For $z$ lying in the support of $dd^c_z\log\max\{\|z-x\|_v,\varepsilon_v\}$, this gives
\begin{align*}
\log\left(\frac{\|s(x)\|_{v,0}}{\|s(z)\|_{v,0}}\right) & 
 \leq \left\{\begin{array}{ll}
\log(1+\varepsilon_v)\leq \varepsilon_v & \text{if} \ v \ \text{is archimedean},\\
0 & \text{if} \ v \ \text{is non-archimedean},
\end{array}\right.
\end{align*}
and the proof is complete.
\end{proof}

\subsection{The mutual energy - height of points relation: proof of Theorem~\ref{tm:split}}\label{sec:split}

Pick integers  $d_1,d_2\geq2$, pick rational maps $f_1\in \mathrm{Rat}_{d_1}(\bar{\mathbb{Q}})$ and $f_2\in \mathrm{Rat}_{d_2}(\bar{\mathbb{Q}})$ and let $\mathbb{K}$ be a number field such that $f_1$ and $f_2$ are both defined over $\mathbb{K}$.

Let $\bar{L}:=\frac{1}{2}(\bar{L}_{f_1}+\bar{L}_{f_2})$, where $\bar{L}_{f_i}$ is the canonical metric of $f_i$ on $\mathcal{O}_{\mathbb{P}^1}(1)$. 

\medskip

Fix $i=1,2$ and apply Theorem~\ref{tm-Green-Holder-adelic}: Given a polynomial lift $F_i$ of $f_i$, defined over a number field $\mathbb{K}$, we have $\mathrm{Res}(F_i)\in \mathbb{K}^\times$. Up to replacing $\mathbb{K}$ with an extension, we can assume there is $\alpha_i\in\mathbb{K}^\times$ such that $\mathrm{Res}(F_i)=\alpha^{2d_i}_i$. We thus may replace $F_i$ by $\alpha_i^{-1}F_i$ to get $\mathrm{Res}(F_i)=1$. For a given $v\in M_\mathbb{K}$, if $X,Y\in \mathbb{A}^{2}(\bar{\mathbb{Q}})\setminus\{0\}$ with $\|X\|_v=\|Y\|_v=1$, by \eqref{resultant}, we have
\[\frac{C_v(d_i)}{\|F_i\|_v^{2d_i}}\leq \frac{\|F_i(X)\|}{\|F_i(Y)\|}\leq C_v(d_i)\|F_i\|_v^{2d_i},\]
with $C_v(d_i)=1$ if $v$ is non-archimedean and $C_v(d_i)\geq1$ depending only on $d$ if $v$ is archimedean. In particular, we have $\|F_i\|_v\geq1$ if $v$ is non-archimedean and $\|F_i\|_v\geq 1/\tilde{C}(d_i)$ 
 when $v$ is non-archimedean, where $\tilde{C}(d_i)^{2d_i}=C(d_i)$.
 
For any $v\in M_\mathbb{K}$, we denote by $u_{F,v}$, $g_{F,v}$ and $\|F\|_v$ the objects defined in Section~\ref{sec:Holder} using the $v$-adic norm.

We let $C_v(F_i)$ and $\alpha_v(F_i)$ be the constants of the H\"older property of $\bar{L}_{f_i}$ at place $v$. By Theorem~\ref{tm-Green-Holder-adelic}, if $v$ is archimedean we have
\[\left\{\begin{array}{ll}
C_v(F_i)=C_3\left(C_{1}+(2d_i+1)\log\|F_i\|_v\right),  &   \text{and} \\ 
\alpha_v(F_i)^{-1}=\frac{1}{\log d_i}\left( C_{2}+\max\{\log(2d_i),2d_i\log\|F_i\|_v\}\right),&
\end{array}\right.\]
where $C_{1},C_{2},C_3\geq1$ depend only on $d$ and, when $v$ is non-archimedean,
\[\left\{\begin{array}{ll}
C_v(F_i)=C_3(2d_i+1)\log\|F_i\|_v,  &   \text{and} \\ 
\alpha_v(F_i)^{-1}=\left\{\begin{array}{ll}
1 & \text{if} \ \|F_i\|_v=1,\\
4d_i\log\|F_i\|_v/\log (d_i) & \text{otherwise}.
\end{array}\right.
&\end{array}\right.\]
For any $v\in M_\mathbb{K}$ with $C_v(F_i)=0$, we have $\alpha_v(F_i)=1$ and $g_{F_i,v}\equiv0$ in this case.

\medskip

By this discussion, the adelic line bundle $\bar{L}$ is adelically H\"older with constants $\{C_v(\bar{L})\}_{v\in M_\mathbb{K}}$ and $\{\alpha_v(\bar{L})\}_{v\in M_\mathbb{K}}$ satisfying
\[\left\{\begin{array}{ll}
1\leq C_v(\bar{L})\leq C_3\left(C_{1}+(2d_1+1)\log\|F_1\|_v+(2d_2+1)\log\|F_2\|_v\right),  &   \text{and} \\ 
\alpha_v(\bar{L})^{-1}\leq\sum_{i=1}^2 \frac{1}{\log d_i}\left( C_{2}+\max\{\log(2d_i),2d_i\log\|F_i\|_v\}\right),&
\end{array}\right.\]
where $C_{1},C_{2},C_3\geq1$ depend only on $d$ when $v$ is archimedean, and
\[\left\{\begin{array}{ll}
1\leq  C_v(\bar{L})\leq C_3((2d_1+1)\log\|F_1\|_v+(2d_2+1)\log\|F_2\|_v),  &   \text{and} \\ 
\alpha_v(\bar{L})^{-1}\leq \left\{\begin{array}{ll}
1 & \text{if} \ \|F_1\|_v=\|F_2\|_v=1,\\
\sum_{i=1}^24d_i\log\|F_i\|_v/\log (d_i) & \text{otherwise},
\end{array}\right.
&\end{array}\right.\]
when $v$ is non-archimedean.

\medskip

Choose a finite Galois-invariant subset $E\subset\mathbb{A}^1(\bar{\mathbb{Q}})$ and a small adelic constant $\{\varepsilon_v\}_{v\in M_\mathbb{K}}$. Combined with Theorem~\ref{tm:holder} and Lemma~\ref{lm:heightpairing}, the above gives
\begin{align}
\langle f_1,f_2\rangle\leq & \frac{1}{\# E}\sum_{x\in E}(\hat{h}_{f_1}(x)+\hat{h}_{f_2}(x))+ \sum_{v\in M_\mathbb{K}}\frac{N_v}{[\mathbb{K}:\mathbb{Q}]}\left(4\log C_v(\bar{L})\varepsilon_v^{\alpha_v(\bar{L})}-\frac{\log(\varepsilon_v)}{\# E}\right).\label{almost-there}
\end{align}
We now choose the small adelic constant $\{\varepsilon_v\}_{v\in M_\mathbb{K}}$.
Fix $0<\delta<1$ and $v\in M_\mathbb{K}$ and set
\[\varepsilon_v:=\left\{\begin{array}{ll}
\delta^{1/\alpha_v(\bar{L})} & \text{if} \ \log C_v(\bar{L})> 0,\\
1 & \text{otherwise}.
\end{array}\right.\]
To conclude the proof, we use \eqref{almost-there} and the choice of $\{\varepsilon_v\}_v$ to get:
\begin{align*}
\mathscr{E}  := &\langle f_1,f_2\rangle -\frac{1}{\# E}\sum_{x\in E}(\hat{h}_{f_1}(x)+\hat{h}_{f_1}(x))\\
& \leq \sum_{v\in M_\mathbb{K}}\frac{N_v}{[\mathbb{K}:\mathbb{Q}]}\left(4C_v(\bar{L})\varepsilon_v^{\alpha_v(\bar{L})}-\frac{\log(\varepsilon_v)}{2\# E}\right)\\
& \leq \sum_{v\in M_\mathbb{K}}\frac{N_v}{[\mathbb{K}:\mathbb{Q}]}4C_v(\bar{L})\delta-\sum_{v\in M_\mathbb{K}, \  C_v(\bar{L})\neq 0}\frac{N_v}{[\mathbb{L}:\mathbb{Q}]}\frac{\log(\delta)}{\alpha_v(\bar{L})\# E}.
\end{align*}
Let $M_\mathbb{K}^\infty$ denote the set of archimedean places of $\mathbb{K}$. The above implies
\begin{align*}
\sum_{v\in M_\mathbb{K}}\frac{N_v}{[\mathbb{K}:\mathbb{Q}]}4C_v(\bar{L}) \leq & A\sum_{v\in M_\mathbb{K}}\frac{N_v}{[\mathbb{L}:\mathbb{Q}]}(\log\|F_1\|_v+\log\|F_2\|_v)\\
&+4C_3\sum_{v\in M_\mathbb{K}^\infty}\frac{N_v}{[\mathbb{K}:\mathbb{Q}]}\log(C_{1})\\
\leq & A\cdot\left( h_{\mathrm{Rat}_{d_1}}(f_1)+h_{\mathrm{Rat}_{d_2}}(f_2)\right)+B,
\end{align*}
where $A=4C_3(d_1+d_2+1)$ and $B:=4C_3\log(C_{1})$. Similarly, there are $A_1\geq A$ anb $B_1\geq B$ depending only on $d_1$ and $d_2$ such that
\begin{align*}
\sum_{v\in M_\mathbb{K}\atop C_v(\bar{L})\neq 0}\frac{N_v}{[\mathbb{K}:\mathbb{Q}]}\frac{1}{\alpha_v(\bar{L})} \leq & \sum_{v\in M_\mathbb{K}\atop C_v(\bar{L})\neq 0}\frac{N_v}{[\mathbb{K}:\mathbb{Q}]}A_1\cdot (\log\|F_1\|_v+\log\|F_2\|_v)\\
&+B_1\sum_{v\in M_\mathbb{K}^\infty}\frac{N_v}{[\mathbb{K}:\mathbb{Q}]}\log(C_{2})\\
\leq & A_1\cdot\left( h_{\mathrm{Rat}_{d_1}}(f_1)+h_{\mathrm{Rat}_{d_2}}(f_2)\right)+B',
\end{align*}
where $B'=B_1\log(C_{2})$. This concludes the proof.

\section{A current and a measure on parameter spaces}

\subsection{The complex pairing and a bifurcation current}\label{sec:mubif}
For the whole section, we fix two integers $d_1,d_2\geq2$ and we consider families of pairs of rational maps $(f_{1,t},f_{2,t})$ of degrees $d_1$ and $d_2$ parametrized by a complex quasi-projective variety $S$. Such a family is given by a morphism
\[\mathcal{F}:S\times \mathbb{P}^1\times\mathbb{P}^1\longrightarrow S\times \mathbb{P}^1\times\mathbb{P}^1\]
such that, for any $t\in S$, $(f_{1,t},f_{2,t})=\mathcal{F}(t,\cdot,\cdot)$ is a pair of rational maps of respective degrees $d_1$ and $d_2$. In what follows, we denote by 
\begin{itemize}
\item $\pi:S\times\mathbb{P}^1\times\mathbb{P}^1\to S$ the canonical projection and 
\item $\pi_i:S\times\mathbb{P}^1\times\mathbb{P}^1\to S\times\mathbb{P}^1$ the projection onto $S$ and the $i$-th factor $\mathbb{P}^1$.
\end{itemize}
For $i=1,2$, following \cite{GV_Northcott}, we let $\widehat{T}_{\mathfrak{f}_i}$ is the fibered Green current of the family $\mathfrak{f}_i:S\times \mathbb{P}^1\to S\times\mathbb{P}^1$ of degree $d$ rational maps induced by $\mathcal{F}$, we let $\widehat{T}_i:=\pi_i^*(\widehat{T}_{\mathfrak{f}_i})$. Let also $\Delta\subset\mathbb{P}^1\times\mathbb{P}^1$ be the diagonal. As in \cite{DeMarco_Mavraki_2024} we define:
\begin{definition}
The \emph{pairwise-bifurcation current} of the family $\mathcal{F}$ is
\[T_{\mathcal{F},\Delta}:=\pi_*\left(\left(\widehat{T}_{1}+\widehat{T}_{2}\right)^{\wedge2}\wedge[S(\C)\times \Delta]\right).\]
\end{definition}

 We also define a function $U:S(\C)\longrightarrow\mathbb{R}_+$ by
\[U(f_{1,t},f_{2,t}):=(\mu_{f_{1,t}}-\mu_{f_{2,t}},\mu_{f_{1,t}}-\mu_{f_{2,t}}), \quad t\in S(\mathbb{C}).\]
\begin{lemma}\label{lm:potentialmu}
The function $U$ is plurisubharmonic and continuous and satisfies
\[T_{\mathcal{F},\Delta}=2\ddc U.\]
\end{lemma}

%

\begin{proof}
Let $\Omega\subset S(\mathbb{C})$ be a small simply connected open subset. This defines two analytic families of rational maps $f_{1,t}$ and $f_{2,t}$ of degree $d_1$ and $d_2$ respectively. Up to reducing $\Omega$, we can define holomorphic families $F_{1,t}$ and $F_{2,t}$ of polynomial homogeneous lifts.
By definition, on has $\mu_{f_{1,t}}-\mu_{f_{2,t}}=\ddc (g_{F_{1,t}}-g_{F_{2,t}})$, so that \cite[\S2]{FRL} gives
\begin{align*}
(\mu_{f_{1,t}}-\mu_{f_{2,t}},\mu_{f_{1,t}}-\mu_{f_{2,t}})& = -\int_{\mathbb{P}^1(\mathbb{C})}(g_{F_{1,t}}-g_{F_{2,t}})\ddc(g_{F_{1,t}}-g_{F_{2,t}}).
\end{align*}
The continuity of $(z,t)\mapsto g_{F_{1,t}}(z)$ and $(z,t)\mapsto g_{F_{2,t}}(z)$ implies that $U$ is continuous on $\Omega$.

We now prove that $U$ is psh and that $\ddc U=\pi_*((\widehat{T}_{1}+\widehat{T}_{2})^{\wedge2}\wedge[S(\C)\times \Delta])$. Let $\omega$ be the standard Fubini-Study form on $\mathbb{P}^1(\mathbb{C})$ and $p_1:S\times\mathbb{P}^1\to S$ and $p_2:S\times\mathbb{P}^1\to\mathbb{P}^1$ be the canonical projections. Recall that $\widehat{T} _i=\lim_n d_i^{-n}(f_i^n)^*(p_i^*\omega)$ and that the local ppotentials converge uniformly. As $(p_i^*\omega)^{\wedge2})=0$ for $i=1,2$, this implies $\widehat{T}_{1}^{\wedge2}=\widehat{T}_{2}^{\wedge2}=0$. Denote by $w(t,z):=g_{F_{1,t}}(z)-g_{F_{2,t}}(z)$ and define a $(1,1)$-current $T_\Omega$ on $\Omega\times \mathbb{P}^1$ by
\[T_\Omega:=- w(\widehat{T}_1-\widehat{T}_2).\]
By definition, since $\widehat{T}_1^2=\widehat{T}_2^2=0$,  we have $\mathrm{dd}^cT_\Omega=\widehat{T}_1\wedge\widehat{T}_2$. By \cite[Proposition 4.3]{BB1},
\[(T_\Omega)_t=-(g_{F_{1,t}}(z)-g_{F_{2,t}}(z))\mathrm{dd}^c(g_{F_{1,t}}-g_{F_{2,t}}),\]
which in turn implies
\[(p_1)_*(T_\Omega)=-\int_{\mathbb{P}^1(\mathbb{C})}(g_{F_{t,1}}-g_{F_{2,t}})\mathrm{dd}^c(g_{F_{1,t}}-g_{F_{2,t}})=(\mu_{f_{1,t}}-\mu_{f_{2,t}},\mu_{f_{1,t}}-\mu_{f_{2,t}}).\]
We thus have $(p_1)_*(T_\Omega)=U$ and for a test form $\psi$ on $\Omega$,
\[\langle\ddc U,\psi\rangle=\int_{\Omega}U\ddc\psi=\int_{\Omega}(p_1)_*(T_\Omega)\ddc\psi=\int_{\Omega}\ddc(p_1)_*(T_\Omega)\wedge\psi,\]
and if $T=(\widehat{T}_{1}+\widehat{T}_{2})^{\wedge2}\wedge[S(\C)\times \Delta]$, we have
\begin{align*}
\langle \pi_*T,\psi\rangle & = \langle T,\pi^*\psi\rangle =2\int_{\Omega}\widehat{T}_{1}\wedge\widehat{T}_2\wedge p_1^*\psi\\
& = 2\int_{\Omega}(p_1)_*(\mathrm{dd}^cT_\Omega)\wedge \psi=2\int_{\Omega}\mathrm{dd}^c(p_1)_*(T_\Omega)\wedge \psi=2\langle\ddc U,\psi\rangle.
\end{align*}
Since $T$ is a positive current, $\ddc U$ is positive. This concludes the proof.
\end{proof}

Assume now that $S=\mathrm{Rat}_{d_1}\times\mathrm{Rat}_{d_2}$. The next lemma is important in what follows.
\begin{lemma}\label{lm:conjug}
For any $(f,g)\in\mathrm{Rat}_{d_1}\times\mathrm{Rat}_{d_2}(\C)$ and any $\varphi\in\mathrm{PGL}(2,\C)$,
\[U(f,g)=U(\varphi^{-1}\circ f\circ\varphi,\varphi^{-1}\circ g\circ\varphi).\]
\end{lemma}
\begin{proof}
Pick $(f,g)\in \mathrm{Rat}_{d_1}\times \mathrm{Rat}_{d_2}(\C)$ and $\varphi\in \mathrm{PGL}(2,\C)$. We use the notations $f^\varphi:=\varphi^{-1}\circ f \circ \varphi$ and  $g^\varphi:=\varphi^{-1}\circ g \circ \varphi$ in this proof. By construction, we have
\begin{align*}
\mu_{f^\varphi}=\lim_{n\to\infty}d^{-n}((f^\varphi)^n)^*\omega=\lim_{n\to\infty}d^{-n}\varphi^*(f^n)^*((\varphi^{-1})^*\omega)=\varphi^*\mu_f.
\end{align*}
In particular, for any choice of lifts $F$ of $f$ and $G$ of $g$ respectively, we have $\mu_{f^\varphi}=\varphi^*(\omega)+\ddc (g_F\circ \varphi)$ and $\mu_{g^\varphi}=\varphi^*(\omega)+\ddc (g_G\circ \varphi)$, whence $\mu_{f^\varphi}-\mu_{g^\varphi}=\ddc (g_F\circ \varphi-g_G\circ \varphi)$. This in turn implies
\begin{align*}
(\mu_{f^\varphi}-\mu_{g^\varphi},\mu_{f^\varphi}-\mu_{g^\varphi}) &= -\int_{\mathbb{P}^1(\mathbb{C})}\varphi^*\left((g_{F}-g_{G})\ddc(g_{F}-g_{G})\right)\\
&= - \int_{\mathbb{P}^1(\mathbb{C})}(g_{F}-g_{G})\ddc(g_{F}-g_{G})\\
& =(\mu_f-\mu_g,\mu_f-\mu_g),
\end{align*}
by the change of variable formula.
\end{proof}

\subsection{The measure on the quotient space}

Pick two integers $d_1,d_2\geq2$. Define $X_{d_1,d_2}:=(\mathrm{Rat}_{d_1}\times \mathrm{Rat}_{d_2})/\mathrm{PGL}(2)$, where the action of $\mathrm{PGL}(2)$ is by simultaneous conjugacy. The variety $X_{d_1,d_2}$ is qusiprojective, irreducible of dimension $2(d_1+d_2)-1$ and defined over $\mathbb{Q}$, see \cite{DeMarco_Mavraki_2024} for more details on the construction of such a space. As the function $U$ is $\mathrm{PGL}(2)$-invariant, it descends to a function
\[u:X_{d_1,d_2}(\C)\longrightarrow\mathbb{R}_+\]
which is psh and continuous on $X_{d_1,d_2}(\mathbb{C})$.
\begin{definition}
The \emph{bifurcation measure} of the family $X_{d_1,d_2}$ is
\[\mu_{d_1,d_2}:=(\ddc u)^{\wedge(2(d_1+d_2)-1)}.\]
\end{definition}
Denote by $\Pi:\mathrm{Rat}_{d-1}\times\mathrm{Rat}_{d_2}\to X_{d_1,d_2}$ the canonical projection. Let also 
\[\mathcal{F}:\mathrm{Rat}_{d_1}\times\mathrm{Rat}_{d_2}\times \mathbb{P}^1\times\mathbb{P}^1\to\mathrm{Rat}_{d_1}\times\mathrm{Rat}_{d_2}\times \mathbb{P}^1\times\mathbb{P}^1\]
be the universal family of pairs of rational maps.

\begin{lemma}\label{lm:munonzero}
For any $d_1,d_2\geq2$, the measure $\mu_{d_1,d_2}$ is non-zero and satisfies
\[\Pi^*(\mu_{d_1,d_2})=\left(\frac{1}{2}T_{\mathcal{F},\Delta}\right)^{\wedge(2(d_1+d_2)-1)}.\]
\end{lemma}
%

\begin{proof}
We first compute $\Pi^*(\mu_{d_1,d_2})$: for a test $(3,3)$-form$\psi$ on $V:=\mathrm{Rat}_{d_1}\times\mathrm{Rat}_{d_2}(\C)$,
\begin{align*}
\langle \Pi^*(\mu_{d_1,d_2}),\psi\rangle & =\int_{V}(\ddc u\circ \Pi)^{\wedge(2(d_1+d_2)-1)}\wedge\psi=\int_{V}(\ddc U)^{\wedge(2(d_1+d_2)-1)}\wedge\psi\\
& = \int_{V}\left(\frac{1}{2}T_{\mathcal{F},\Delta}\right)^{\wedge(2(d_1+d_2)-1)}\wedge \psi,
\end{align*} 
where we used Lemma~\ref{lm:conjug} and Lemma~\ref{lm:potentialmu} successively.

We now show $\mu_{d_1,d_2}$ is non-zero. First, assume $d_1$ and $d_2$ are multiplicatively independent. Let $t_0\in X_{d_1,d_2}(\mathbb{C})$ be the equivalence class of the $(T_{d_1},T_{d_2})$, where $T_{d_i}$ is the Chebychev degree $d_i$ polynomial. Note that the point $t_0$ is a smooth point if $X_{d_1,d_2}$. Indeed, as $X_{d_1,d_2}$ is the geometric quotient $(\mathrm{Rat}_{d_1}\times\mathrm{Rat}_{d_2})//\mathrm{PGL}(2)$, the singularities of $X_{d_1,d_2}$ arise from points in $\mathrm{Rat}_{d_1}\times\mathrm{Rat}_{d_2}$ with non-trivial stabilizer, i.e. paire $(f,g)\in \mathrm{Rat}_{d_1}\times\mathrm{Rat}_{d_2}$ with an automorphism $\varphi\in \mathrm{PGL}(2)\setminus\{\mathrm{id}\}$ with $\varphi \circ f=f\circ \varphi$ and $\varphi \circ g=g \circ \varphi$. Now, the group of automorphisms $\psi\in \mathrm{PGL}(2)$ with $\psi\circ T_{d_i}=T_{d_i}\circ \psi$ is reduced to the identity (see~e.g.~\cite[\S~3.1]{book-unlikely}), and we proved that $X_{d_1,d_2}$ is smooth at $t_0$.
By the first point of Lemma~\ref{lm:samemu}, the function $u$ has an isolated minimum at $t_0$. As $u$ is psh continuous and $\mu_{d_1,d_2}=(2\ddc u)^{\wedge(2(d_1+d_2)-1)}$, by Lemma~\ref{lm:potentialmu}, $t_0\in \mathrm{supp}(\mu_{d_1,d_2})$.

We now assume there are $n,m\geq1$ such that $d_1^n=d_2^m$. If $n=m=1$, this is proved in \cite{DeMarco_Mavraki_2024}. We thus assume $nm>1$ and let $D:=d^n_1=d_2^m$. The map
\[\iota:\{(f_1,f_2)\}\in X_{d_1,d_2}\longmapsto \{(f_1^n,f_2^m)\}\in X_{D,D}\]
is proper and finite by \cite{DeMarco-boundary}.  Whence $\mu_{d_1,d_2}$ is proportional to $\iota^*(\tilde T)$, where
\[\tilde{T}:=\pi_*\left(\left(\widehat{T}_{1,D}+\widehat{T}_{2,D}\right)^{\wedge2}\wedge[X_{D,D}\times \Delta]\right)^{\wedge(2(d_1+d_2)-1)}\]
and it vanishes if and only if $\tilde{T}$ does.
By \cite[Theorem~7.1]{DeMarco_Mavraki_2024} and the criterion established in \cite{DeMarco_Mavraki_2024} and \cite{sparsity}, we have $\tilde{T}\neq0$. Indeed, DeMarco and Mavraki provide a family of pairs of rational maps parametrized by a quasiprojective variety $S_D$ such that the canonical projection $p_D:S_D\to X_{D,D}$ is finite-to-one onto a dense Zariski open subset $V\subset X_{D,D}$. Moreover, they show that the pairwise-bifurcation measure $\nu$ of this family is non-zero. As $\nu$ has continuous potential, its support is not contained in a proper subvariety of $S_D$ and $p_D^*(\tilde T)=\nu$, whence $\tilde T$ is non-zero.
The proof is complete.
\end{proof}

\section{Common preperiodic points for rational maps}

\subsection{Height inequalities and uniformity}
In this paragraph, we let $X$ and $S$ be quasi-projective varieties defined over a number field $\mathbb{K}$. We assume $\pi:X\to S$ is a family of curves defined over $\mathbb{K}$, i.e. $\pi$ is a projective and flat morphism with relative dimension $1$ whose fibers are geometrically connected. Pick an embedding $\iota:X\hookrightarrow \mathbb{P}^N\times S$ so that $\iota|_{X_t}$ is an embedding of $X_t$ into $\mathbb{P}^N$. Let $D$ be the divisor $D:=\iota^{-1}(H_\infty)$ and assume that $L=\mathcal{O}(D)$ Let $D_t:=D\cap X_t$ for $t\in S$. The line bundle $L$ is relatively ample and we endow $L$ with an semi-positive continuous metrization $\{\|\cdot\|_v\}_{v\in M_\mathbb{K}}$ and denote $\bar{L}=(L,\{\|\cdot\|_v\}_{v\in M_\mathbb{K}})$. We assume that
\begin{enumerate}
\item[$(a)$] for any $t\in S(\bar{\mathbb{K}})$, the metrized line bundle $\bar{L}_t:=\bar{L}|_{X_t}$ is adelic,
\item[$(b)$] $h_{\bar{L}}\geq0$ on $X(\bar{\mathbb{K}})$.
\end{enumerate}
The following is inspired by DeMarco, Krieger and Ye ~\cite{DKY1,DKY2}.

\begin{lemma}\label{lm:uniform-follows-from-ineg}
Let $B$ be a projective model of $S$ and $M$ be an ample line bundle on $B$. Assume there are constants $C>0$ and $C'\geq0$ such that 
\begin{align}
 h_{\bar{L}_t}(X_t)\geq C h_{B,M}(t)-C', \quad \text{for all} \ t\in S(\bar{\mathbb{K}})\label{fundamental-ineg}
\end{align}
and that there are constant $C_1\geq1$ and $C_2\geq0$ such that for any finite $\mathrm{Gal}(\bar{\mathbb{K}}/\mathbb{K})$-invariant set $E\subset X_t(\bar{\mathbb{K}})\setminus\mathrm{supp}(D_t)$, and any $0<\delta<1$, then
\begin{align}
h_{\bar{L}_t}(X_t)\leq C_1\left(\delta-\frac{\log(\delta)}{\# E}\right)\left(h_{B,M}(t)+C_2\right)+\frac{1}{\# E}\sum_{x\in E}h_{\bar{L}_t}(x).\label{energy-cardinal-ineg}
\end{align}
Then the following properties are equivalent:
\begin{enumerate}
\item there exists $\varepsilon>0$ such that for all $t\in S(\bar{\mathbb{K}})$, we have $h_{\bar{L}_t}(X_t)\geq\varepsilon$,
\item there exist $\epsilon>0$ and an integer $N\geq1$ such that, for all $t\in S(\bar{\mathbb{K}})$, 
\[\# \{ x\in X_t(\bar{\mathbb{K}})\, : \ h_{\bar{L}_t}(x)\leq \epsilon\}\leq N.\]
\end{enumerate}
\end{lemma}

\begin{proof}
The implication $2.$ implies $1.$ follows from Zhang's inequalities: they give 
\[h_{\bar{L}_t}(X_t)\geq \frac{1}{2}\left(e_1(\bar{L}|_{X_t})+\inf_{x\in X_t(\bar{\mathbb{K}})}h_{\bar{L}}(x)\right),\]
if $e_1(\bar{L}|_{X_t})$ is the essential infimum of $h_{\bar{L}}$ on $X_t(\bar{\mathbb{K}})$.
If we assume there exists $\epsilon>0$ and $N\geq1$ such that $\# \{ x\in X_t(\bar{\mathbb{K}})\, : \ h_{\bar{L}_t}(x)\leq \varepsilon\}\leq N$ for all $t\in S(\bar{\mathbb{K}})$, since $h_{\bar{L}}\geq0$, this gives $h_{\bar{L}_t}(X_t)\geq \varepsilon:=\epsilon/2$, for all $t\in S(\bar{\mathbb{K}})$.

\medskip

To prove the converse implication, we may use the inequality \eqref{fundamental-ineg}. We assume there is a finite $\mathrm{Gal}(\bar{\mathbb{K}}/\mathbb{K})$-invariant set $E\subset X_t(\bar{\mathbb{K}})\setminus \mathrm{supp}(D_t)$ which is non-empty and such that $h_{\bar{L}}(x)\leq \varepsilon/2$ for all $x\in E$.  
First, we assume $h_{B,M}(t)$ is large. More precisely, we assume $h_{B,M}(t)\geq 4(C'+CC_2+\varepsilon/2)/C+C_2$. Combining \eqref{fundamental-ineg} with ~\eqref{energy-cardinal-ineg} yields
\[ C h_{B,M}(t)-C'\leq h_{\bar{L}_t}(X_t)\leq C_1\left(\delta-\frac{\log(\delta)}{\# E}\right)\left(h_{B,M}(t)+C_2\right)+\frac{\varepsilon}{2},\]
for any $0<\delta<1$. This rewrites as
\[ C-C_1\delta+C_1\frac{\log(\delta)}{\# E}\leq \frac{C'+CC_2+\varepsilon/2}{h_{B,M}(t)-C_2}.\]
Without loss of generalities, we can assume $C\leq 1$. For $\delta=C/(2C_1)<1$, this gives
\[ \frac{C}{2}+C_1\frac{\log(C/(2C_1))}{\# E}\leq \frac{C'+CC_2+\varepsilon/2}{h_{B,M}(t)-C_2}.\]
If we assume $h_{B,M}(t)\geq 4(C'+CC_2+\varepsilon/2)/C+C_2$, we deduce
\[\frac{C}{4}\leq -C_1\frac{\log(C/(2C_1))}{\# E}=C_1\frac{\log(2C_1/C)}{\# E}.\]
We thus have proved $\# E\leq 4C_1/C\cdot\log(2C_1/C)$.

\medskip

We now assume $h_{B,M}(t)\leq 4(C'+CC_2+\varepsilon/2)/C+C_2$. Then~\eqref{energy-cardinal-ineg} yields
\begin{align*}
\varepsilon\leq h_{\bar{L}_t}(X_t) & \leq C_1\left(\delta-\frac{\log(\delta)}{\# E}\right)\left(h_{B,M}(t)+C_2\right)+\frac{\varepsilon}{2}\\
& \leq  C_1\left(\delta-\frac{\log(\delta)}{\# E}\right)\left(\frac{4(C'+CC_2+\varepsilon/2)}{C}+2C_2\right)+\frac{\varepsilon}{2},
\end{align*}
for any $0<\delta<1$. Letting $B:=C_1(4(C'+CC_2+\varepsilon/2)/C+2C_2)$ and $\delta:=\varepsilon/4B$, we deduce
\[\# E\leq -\frac{\log(\delta)}{\varepsilon/2B-\delta}=\frac{4B\log(4B/\varepsilon)}{\varepsilon}.\]
Taking $N\geq \max\{4B/\varepsilon\cdot \log(4B/\varepsilon), 4C_1/C\cdot\log(2C_1/C)\}+\deg_{L_t}(X_t)$ and letting $\epsilon:=\varepsilon/2$, we conclude the proof.
\end{proof}

\subsection{The fundamental height inequality and equidistribution}
Pick any integers $d_1,d_2\geq2$. We say that a sequence $t_n=(f_n,g_n)\in X_{d_1,d_2}(\bar{\mathbb{Q}})$ is called
\begin{itemize}
\item \emph{generic} if for any strict subvariety $Z\subset X_{d_1,d_2}$ that is defined over $\mathbb{Q}$, there is $n_0$ such that $Z\cap (\mathrm{Gal}\cdot t_n)=\varnothing$ for all $n\geq n_0$.
\item \emph{small} if $\langle f_n,g_n\rangle\to0$, as $n\to\infty$.
\end{itemize}
Inspiring from the polarized case $d_1=d_2$, one can derive the following from \cite{YZ-adelic}.
\begin{theorem}\label{tm:equi}
The following hold:
\begin{enumerate}
\item for any ample divisor $N$ on a projective model of $X_{d_1,d_2}$, there are $C_1,C_2>0$ and Zariski open and dense subset $U\subseteq X_{d_1,d_2}$ such that
\[\langle f_1,f_2\rangle\geq C_1 \cdot h_{N}(t)-C_2, \quad (f_1,f_2)\in t\in U(\bar{\mathbb{Q}}).\]
\item If there is a generic and small sequence $(t_n)_n\in X_{d_1,d_2}(\bar{\mathbb{Q}})$, then the sequence
\[\mu_{t_n}=\frac{1}{\#\mathrm{Gal}\cdot t_n}\sum_{t\in \mathrm{Gal\cdot t_n}}\delta_t\]
converges to $\mu_{d_1,d_2}/\mu_{d_1,d_2}(X_{d_1,d_2}(\mathbb{C}))$ in the weak sense of measures on $X_{d_1,d_2}(\mathbb{C})$.
\end{enumerate}
\end{theorem}

\begin{proof}
Pick $i=1,2$. By \cite[Theorem 6.1.1]{YZ-adelic}, the metrized line bundle $\bar{L}_{f_i}$ is a nef adelic line bundle on the quasi-projective variety $\mathbb{P}^1_{\mathrm{Rat}_{d_i}}$. Let $S:=\mathrm{Rat}_{d_1}\times \mathrm{Rat}_{d_2}$ and denote by $\bar{L}_i$ the nef adelic line bundle induced by $\bar{L}_{f_i}$ on $\mathbb{P}^1_S$ by pullback and let $\bar{L}:=\pi_1^*\bar{L}_1+\pi_2^*\bar{L}_2$, where $\pi_i:S\times(\mathbb{P}^1)^2\to S\times\mathbb{P}^1$ is as above. Then $\bar{L}$ is a nef adelic line bundle on $S\times(\mathbb{P}^1)^2$ and it induces a nef adelic line bundle on $X_{d_1,d_2}\times (\mathbb{P}^1)^2$. Denote by
\[\bar{M}:=\langle \bar{L}|X_{d_1,d_2}\times \Delta\rangle^2.\]
By \cite[Theorem~4.1.3]{YZ-adelic}, $\bar{M}$ is a nef adelic line bundle on $X_{d_1,d_2}$.
By Lemma~\ref{lm:potentialmu}, the equilibrium measure of the adelic line bundle is proportional to $\mu_{d_1,d_2}$ and is non-zero.
According to \cite[Lemma~5.4.4]{YZ-adelic}, we deduce that $\deg_{\widetilde{M}}(X_{d_1,d_2})>0$, i.e. $\bar{M}$  is non-degenerate in the sense of Yuan and Zhang. 

We thus can apply the second item of Theorem 5.3.5 from \cite{YZ-adelic} to get the existence, for any adelic line bundle $\bar{N}$ on $X_{d_1,d_2}$, of constants $C_1,C_2>0$ and a dense Zariski open subset $U\subseteq X_{d_1,d_2}$ such that
\[h_{\bar{M}}(t)\geq C_1 h_{\bar{N}}(t)-C_2, \quad t\in U(\bar{\mathbb{Q}}).\]
We also can apply Theorem 5.4.3 from \cite{YZ-adelic} to obtain that for any generic sequence $t_n$ with $h_{\bar{M}}(t_n)\to0$, then $\mu_{t_n}$ converges to $\mu_{d_1,d_2}/\mu_{d_1,d_2}(X_{d_1,d_2}(\mathbb{C}))$ in the weak sense of measures on $X_{d_1,d_2}(\mathbb{C})$.

Pick $\phi\in \mathrm{GL}_2(\bar{\mathbb{Q}})$ and let $(f,g)\in (\mathrm{Rat}_{d_1}\times\mathrm{Rat}_{d_2})(\bar{\mathbb{Q}})$. Proceeding as in the proof of Lemma~\ref{lm:conjug}, we find
\[\langle \phi^{-1}\circ f\circ\phi, \phi^{-1}\circ g\circ\phi\rangle=\langle f,g\rangle,\]
so that, by construction, for all $t=\{(f,g)\}\in X_{d_1,d_2}(\bar{\mathbb{Q}})$, define
\[h_{\bar{M}}(t)=h_{\bar{L}}(\Delta)=\frac{(\bar{L}_t^2|\Delta)}{2\deg_{L}(\Delta)}=\frac{((\bar{L}_{f}+\bar{L}_{g})^2|\Delta)}{4}=h_{\bar{L}_{f}+\bar{L}_{g}}(\mathbb{P}^1)=\langle f,g\rangle.\]
This concludes the proof.
\end{proof}
%
%
%
%
%
%
%

\subsection{Proof of Theorem~\ref{tm:uniform}}\label{sec:main_result}
We now conclude the proof of Theorem~\ref{tm:uniform}. Assume there is a Zariski dense sequence $t_n\in X_{d_1,d_2}(\bar{\mathbb{Q}})$ which is small. Up to extracting a subsequence, we can assume $(t_n)$ is actually generic. By the second point of Theorem~\ref{tm:equi}, for any continuous compactly supported function $\varphi\in\mathscr{C}^0_c(X_{d_1,d_2}(\C),\R)$, we have
\begin{align*}
\lim_{n\to\infty}\int_{X_{d_1,d_2}(\C)}\varphi\, \mathrm{d}\mu_{t_n}=\int_{X_{d_1,d_2}(\C)}\varphi \,\mathrm{d}\mu,
\end{align*}
where $\mu=\mu_{d_1,d_2}/\mu_{d_1,d_2}(X_{d_1,d_2}(\C))$ and $\mu_{t_n}=\frac{1}{\#\mathrm{Gal}\cdot t_n}\sum_{t\in \mathrm{Gal\cdot t_n}}\delta_t$. Take $t_0\in \mathrm{supp}(\mu)$ and pick a cut-off function $\psi\in\mathscr{C}^0_c(X_{d_1,d_2}(\C),\R_+)$ such that $0\leq \psi\leq 1$ and such that there is an euclidean neighborhood $\Omega\subset X_{d_1,d_2}(\C)$ with $\psi\equiv1$ on $\Omega$.  Recall that $u$ is the potential of $\mu$ defined in Section~\ref{sec:mubif} and that $u\geq0$ and set $\varphi:=u\cdot \psi$. Then $\varphi\in\mathscr{C}^0_c(X_{d_1,d_2}(\C),\R_+)$ and $\varphi=u$ on $\Omega$. By Lemma~\ref{lm:samemu}, the set $E=\{t\in X_{d_1,d_2}(\mathbb{C})$; $u(t)=0\}$ is contained in a pluripolar set. Now, as we chose $\Omega$ to be a neighborhood of $t_0\in \mathrm{supp}(\mu)$, we have 
\[\mu(\Omega\setminus E)=\mu(\Omega)>0,\]
since $\mu$ has contiunous potentials. Whence
\[\int_{X_{d_1,d_2}(\C)}\varphi \, \mathrm{d}\mu\geq\int_\Omega u \, \mathrm{d}\mu=\int_{\Omega\cap\{u>0\}} u\, \mathrm{d}\mu>0.\]
Now, as we chose $\Omega$ to be a neighborhood of $t_0\in \mathrm{supp}(\mu)$, we have $\mu(\Omega)>0$ and, as $\mu_{t_n}\to\mu$ and $\Omega$ is open, this implies we have $\mu_{t_n}(\Omega)\geq \mu(\Omega)/2>0$ for all $n$ large enough. 
Also, by non-negativity of the local mutual energy pairing (see~\cite{FRL}), we have
\[\int_{X_{d_1,d_2}(\C)}u\, \mathrm{d}\mu_{t_n}\leq\frac{1}{\#\mathrm{Gal}\cdot t_n}\sum\limits_{(f,g)\in\mathrm{Gal}\cdot t_n}\langle f,g\rangle=\langle f_n,g_n\rangle.\]
The hypothesis that the sequence $t_n$ is small implies that
\[\lim_{n\to\infty}\int_{X_{d_1,d_2}(\C)}u \, \mathrm{d}\mu_{t_n}=0.\]
As $0\leq \varphi\leq u$, this implies
\[\int_{X_{d_1,d_2}(\C)}\varphi \, \mathrm{d}\mu=\lim_{n\to\infty}\int_{X_{d_1,d_2}(\C)}\varphi \, \mathrm{d}\mu_{t_n}=0.\]
This is a contradiction and we have proved the existence of $\varepsilon>0$ and of a Zariski open an dense subset $U\subseteq X_{d_1,d_2}$ such that
\[\langle f,g\rangle\geq\varepsilon>0 , \quad (f,g)\in U(\bar{\mathbb{Q}}).\]
The conclusion follows from the combination of Theorem~\ref{tm:split} (combined with Lemma 17 from \cite{ingram-explicit}), Lemma~\ref{lm:heightpairing}, Lemma~\ref{lm:uniform-follows-from-ineg} and the first point of Theorem~\ref{tm:equi}.
\bibliographystyle{short} 
\bibliography{biblio}

\begin{thebibliography}{DKY2}

\bibitem[AY]{Shang-Yap}
Yan~Sheng Ang and Jit~Wu Yap.
\newblock Shared {Dynamically}-{Small} {Points} for {Polynomials} on {Average}.
\newblock Preprint, {arXiv}:2312.05115 [math.{DS}] (2023), 2023.

\bibitem[BB]{BB1}
Giovanni Bassanelli and Fran{\c{c}}ois Berteloot.
\newblock Bifurcation currents in holomorphic dynamics on {$\mathbb{P}^k$}.
\newblock {\em J. Reine Angew. Math.}, 608:201--235, 2007.

\bibitem[CL]{ACL2}
Antoine Chambert-Loir.
\newblock Heights and measures on analytic spaces. {A} survey of recent
  results, and some remarks.
\newblock In {\em Motivic integration and its interactions with model theory
  and non-{A}rchimedean geometry. {V}olume {II}}, volume 384 of {\em London
  Math. Soc. Lecture Note Ser.}, pages 1--50. Cambridge Univ. Press, Cambridge,
  2011.

\bibitem[D]{DeMarco-boundary}
Laura DeMarco.
\newblock Iteration at the boundary of the space of rational maps.
\newblock {\em Duke Math. J.}, 130(1):169--197, 2005.

\bibitem[DKY1]{DKY1}
Laura DeMarco, Holly Krieger, and Hexi Ye.
\newblock Uniform {M}anin-{M}umford for a family of genus 2 curves.
\newblock {\em Ann. of Math. (2)}, 191(3):949--1001, 2020.

\bibitem[DKY2]{DKY2}
Laura DeMarco, Holly Krieger, and Hexi Ye.
\newblock Common preperiodic points for quadratic polynomials.
\newblock {\em Journal of Modern Dynamics}, 18(0):363--413, 2022.

\bibitem[DM]{DeMarco_Mavraki_2024}
Laura DeMarco and Niki~Myrto Mavraki.
\newblock Dynamics on $\mathbb{P}^1$: preperiodic points and pairwise
  stability.
\newblock {\em Compositio Mathematica}, 160(2):356--387, 2024.

\bibitem[DS]{dinhsibony2}
Tien-Cuong Dinh and Nessim Sibony.
\newblock Dynamics in several complex variables: endomorphisms of projective
  spaces and polynomial-like mappings.
\newblock In {\em Holomorphic dynamical systems}, volume 1998 of {\em Lecture
  Notes in Math.}, pages 165--294. Springer, Berlin, 2010.

\bibitem[F]{Fili}
Paul Fili.
\newblock A metric of mutual energy and unlikely intersections for dynamical
  systems, 2017.

\bibitem[FG]{book-unlikely}
Charles Favre and Thomas Gauthier.
\newblock Unlikely intersection for polynomial dynamical pairs, 2019.
\newblock preprint.

\bibitem[FRL]{FRL}
C.~Favre and J.~Rivera-Letelier.
\newblock Equidistribution quantitative des points de petite hauteur sur la
  droite projective.
\newblock {\em Math. Ann.}, 335(2):311--361, 2006.

\bibitem[GTV]{sparsity}
Thomas Gauthier, Johan Taflin, and Gabriel Vigny.
\newblock Sparsity of postcritically finite maps of $\mathbb{P}^k$ and beyond:
  A complex analytic approach, 2023.

\bibitem[GV]{GV_Northcott}
Thomas Gauthier and Gabriel Vigny.
\newblock The {G}eometric {D}ynamical {N}orthcott and {B}ogomolov {P}roperties,
  2019.
\newblock To appear in \emph{Ann. Sci. Ecole Nrom. Sup.}

\bibitem[I]{ingram-explicit}
Patrick Ingram.
\newblock Explicit canonical heights for divisors relative to endomorphisms of
  $\mathbb{P}^n$, 2022.

\bibitem[KS]{Kawaguchi-Silverman-Green}
Shu Kawaguchi and Joseph~H. Silverman.
\newblock Nonarchimedean {G}reen functions and dynamics on projective space.
\newblock {\em Math. Z.}, 262(1):173--197, 2009.

\bibitem[LP]{Levin-Przytycki}
G.~Levin and F.~Przytycki.
\newblock When do two rational functions have the same {Julia} set?
\newblock {\em Proc. Am. Math. Soc.}, 125(7):2179--2190, 1997.

\bibitem[MS]{mavraki-schmidt}
Niki~Myrto Mavraki and Harry Schmidt.
\newblock On the dynamical bogomolov conjecture for families of split rational
  maps, 2022.

\bibitem[P]{Poineau-BFT}
J{\'e}r{\^o}me {Poineau}.
\newblock {Dynamique analytique sur $\mathbf{Z}$. II : {\'E}cart uniforme entre
  Latt{\`e}s et conjecture de Bogomolov-Fu-Tschinkel}.
\newblock {\em arXiv e-prints}, page arXiv:2207.01574, July 2022.

\bibitem[PST]{szpiro}
Clayton Petsche, Lucien Szpiro, and Michael Tepper.
\newblock Isotriviality is equivalent to potential good reduction for
  endomorphisms of {$\mathbb{P}^N$} over function fields.
\newblock {\em J. Algebra}, 322(9):3345--3365, 2009.

\bibitem[YZ1]{YZ-Hodge}
Xinyi Yuan and Shou-Wu Zhang.
\newblock The arithmetic {H}odge index theorem for adelic line bundles.
\newblock {\em Math. Ann.}, 367(3-4):1123--1171, 2017.

\bibitem[YZ2]{YZ-adelic}
Xinyi Yuan and Shou-Wu Zhang.
\newblock Adelic line bundles on quasi-projective varieties.
\newblock Preprint, {arXiv}:2105.13587 [math.{NT}] (2021), 2021.

\bibitem[Z]{Zhang-positivity}
Shouwu Zhang.
\newblock Positive line bundles on arithmetic varieties.
\newblock {\em J. Amer. Math. Soc.}, 8(1):187--221, 1995.

\end{thebibliography}
\end{document}